\documentclass[reqno]{amsart}
\DeclareFontSeriesDefault[rm]{bf}{b}\numberwithin{equation}{section}

\usepackage{amsmath}
\usepackage{amsfonts}
\usepackage{amsthm}
\usepackage{amssymb}
\usepackage{graphicx}
\usepackage{graphics}
\usepackage{bm}
\usepackage{dsfont}
\usepackage{color}
\usepackage[font=footnotesize]{caption}
\usepackage[all]{xy}
\usepackage{tikz-cd}
\usepackage{mathtools}
\usepackage{scrextend}
\usepackage{todonotes}
\usepackage{url}
\usepackage{enumitem, hyperref}\hypersetup{colorlinks}

\definecolor{darkred}{rgb}{1,0,0}
\definecolor{darkgreen}{rgb}{0,0.8,0}
\definecolor{darkblue}{rgb}{0,0,1}
\hypersetup{colorlinks,
linkcolor=darkblue,
filecolor=darkgreen,
urlcolor=darkred,
citecolor=darkgreen}

\newtheorem{thm}{Theorem}[section]

\newtheorem{lem}[thm]{Lemma}
\newtheorem{prop}[thm]{Proposition}

\newtheorem{maintheorem}{Theorem}
\theoremstyle{definition}
\newtheorem{rem}[thm]{Remark}

\newcommand{\CP}{\C\PP}
\newcommand{\G}{\mathcal{G}}

\newcommand{\HH}{\mathcal{H}}

\newcommand{\Sp}{\mathrm{Sp}}
\newcommand{\Sym}{\mathrm{Sym}}
\newcommand{\fix}{\mathrm{fix}}

\newcommand{\A}{\mathcal{A}}

\newcommand{\N}{\mathds{N}}

\newcommand{\Z}{\mathds{Z}}
\newcommand{\R}{\mathds{R}}
\newcommand{\Q}{\mathds{Q}}

\newcommand{\PP}{\mathds{P}}
\newcommand{\C}{\mathds{C}}

\newcommand{\crit}{\mathrm{crit}}
\newcommand{\diff}{\mathrm{d}}

\newcommand{\odd}{\mathrm{odd}}
\newcommand{\ind}{\mathrm{ind}}
\newcommand{\id}{\mathrm{id}}
\newcommand{\nul}{\mathrm{nul}}

\newcommand{\tor}{\mathrm{tor}}
\newcommand{\Tan}{\mathrm{T}}

\newcommand{\equaldown}{\arrow[d, no tail, no head, shift left =0.4]\arrow[d, no tail, no head, shift right =0.4]}
\renewcommand{\twoheadrightarrow}{\mathrel{\mathrlap{\rightarrow}\mkern-2.3mu\rightarrow}}

\DeclareRobustCommand{\llongrightarrow}{\relbar\joinrel\relbar\joinrel\rightarrow}
\DeclareRobustCommand{\longepi}{\relbar\joinrel\twoheadrightarrow}
\DeclareRobustCommand{\llongepi}{\relbar\joinrel\relbar\joinrel\twoheadrightarrow}
\DeclareMathOperator{\rank}{\mathrm{rank}} 
 
\DeclareMathOperator*{\toup}{\longrightarrow} 
\DeclareMathOperator*{\ttoup}{\llongrightarrow}

\DeclareMathOperator*{\epi}{\longepi} 
\DeclareMathOperator*{\eepi}{\llongepi}

\begin{document}

\title{On the structure of Besse convex contact spheres}

\author{Marco Mazzucchelli}
\address{Marco Mazzucchelli\newline\indent CNRS, UMPA, \'Ecole Normale Sup\'erieure de Lyon, 69364 Lyon, France}
\email{marco.mazzucchelli@ens-lyon.fr}

\author{Marco Radeschi}
\address{Marco Radeschi\newline\indent Department of Mathematics, University of Notre Dame, IN 46556, USA}
\email{mradesch@nd.edu}

\date{December 10, 2020}



\begin{abstract}

We consider convex contact spheres $Y$ all of whose Reeb orbits are closed. Any such $Y$ admits a stratification by the periods of closed Reeb orbits. We show that $Y$ ``resembles'' a contact ellipsoid: any stratum of $Y$ is an integral homology sphere, and the sequence of Ekeland-Hofer spectral invariants of $Y$ coincides with the full sequence of action values, each one repeated according to its multiplicity. 
\end{abstract}

\maketitle

\vspace{-5pt}

\section{Introduction}
\label{s:introduction}

We recall that a contact manifold is a $(2n-1)$-dimensional manifold $Y$ equipped with a one-form $\alpha$, called the contact form, such that $\alpha\wedge(\diff\alpha)^{n-1}$ is nowhere vanishing. The associated Reeb flow $\psi^t:Y\to Y$ is obtained by integrating the Reeb vector field $R$, which is defined by $\alpha(R)\equiv1$ and $\diff\alpha(R,\cdot)\equiv0$. Reeb flows generalize, for instance, the geodesic flows on the unit tangent bundle of Riemannian manifolds.

A contact manifold $(Y,\alpha)$ is called \textbf{Besse} when every orbit of the Reeb flow is closed. In the Riemannian case, this reduces to the notion of Riemannian manifold all of whose geodesics are closed, a topic of great interest in Riemannian geometry, see for example the influential monograph by Besse \cite{Besse:1978pr} after whom these spaces are named, and more recent developments \cite{Pries:2009aa,Lin:2017aa,Radeschi:2017dz,Abbondandolo:2017xz,Mazzucchelli:2018aa,Mazzucchelli:2018pb}.
In the general setting, Besse manifolds were first studied by Thomas \cite{Thomas:1976aa}, who extended earlier work of Boothby-Wang \cite{Boothby:1958ss} and called them almost regular contact manifolds.

In this paper, we focus on contact manifolds $Y$ that are convex spheres. These are the smooth boundaries of convex compact neighborhoods $B\subset\R^{2n}$ of the origin, and come naturally equipped with the contact form $\alpha=\lambda|_Y$ that is the restriction of the 1-form on $\R^{2n}$ given by
\begin{align*}
 \lambda=\frac12\sum_{i=1}^n \Big( x_i\,\diff y_i - y_i\,\diff x_i \Big).
\end{align*}
The only known examples of Besse convex contact spheres are the rational ellipsoids
\begin{align*}
 E(a_1,\ldots,a_n)
 =
 \left\{ 
 z=(z_1,\ldots,z_n)\in\C^{n}\ 
 \bigg|\ 
 \sum_{i=1}^n\frac{|z_i|^2}{a_i}=\frac1\pi
 \right\},
\end{align*}
where $a_1\leq a_2\leq \ldots\leq a_n$ are positive real numbers with $\tfrac{a_i}{a_j}\in \Q$ for all $i, j=1\ldots n$.
The results in this paper show that a general Besse convex contact sphere $Y$ always ``resembles'', in several aspects, a rational ellipsoid.

The first resemblance concerns the stratification induced by the Reeb flow. We recall that, on any Besse contact manifold $Y$,
a theorem due to Wadsley \cite{Wadsley:1975sp} guarantees that the Reeb flow $\psi^t$ is itself periodic, that is, $\psi^\tau=\id$ for some minimal common period $\tau>0$. Therefore, $\psi^t$ defines a locally free circle action on $Y$. However, different orbits may have different minimal periods: for any positive integer $k>0$, the subspace
\begin{align*}
 Y_k:=\fix(\psi^{\tau/k})=\big\{ z\in Y\ \big|\ \psi^{\tau/k}(z)=z \big\}
\end{align*}
is a (possibly empty) closed contact submanifold of $Y$. A priori, $Y_k$ may be disconnected, and different connected components may have different dimension. The family of subspaces $Y_k$ defines a stratification of $Y$.

On the ellipsoid $E=E(a_1,\ldots,a_n)\subset\R^{2n}\equiv\C^n$, the Reeb flow is given by
\begin{align*}
 \psi^t(z_1,\ldots,z_n)=(e^{i2\pi t/a_1}z_1,\ldots,e^{i2\pi t/a_n}z_n).
\end{align*}
If $E$ is a rational ellipsoid, the least common period $\tau$ of the periodic Reeb orbits is precisely the least common multiple of $a_1,\ldots,a_n$, and the strata $E_k=\fix(\psi^{\tau/k})$ are the sub-ellipsoids
\[
 E_k=\big\{ (z_1,\ldots,z_n)\in E\ \big|\ z_i=0\mbox{ if }\tfrac\tau {k a_i}\not\in\N \big\}.
\]
Our first result describes the topology of the strata of general convex contact spheres.

\begin{maintheorem}\label{t:main}
Let $Y$ be a Besse convex contact sphere. For each $k\in\N$, the stratum $Y_k$ is either empty or an integral homology sphere, i.e.\ $H_*(Y_k;\Z)\cong H_{*}(S^{2d+1};\Z)$, where $2d+1=\dim (Y_k)$.
\end{maintheorem}

To the best of our knowledge, it is unknown whether the same result is true for a general locally free $S^1$ action on a sphere $Y$. The classical Smith theorem \cite[Th.~III.5.1]{Bredon:1972aa} implies that, if $k$ is a power of a prime $p$, then the subspace $Y_k$, being the set fixed by the cyclic subgroup $\Z_k\subset S^1$, is either empty or a $\Z_p$ homology sphere (i.e.\ $H_*(Y_k;\Z_p)\cong H_*(S^{d};\Z_p)$ for some $d\geq0$). On the other hand, there are examples of manifolds equipped with a
smooth action by a finite cyclic group whose fixed point set is \emph{not} an integral homology sphere \cite{Schultz:1982aa}, although it is not known if these examples can arise as strata of smooth locally free circle actions.

The second resemblance to ellipsoids regards the Ekeland-Hofer spectral invariants \cite{Ekeland:1987aa}
\[c_0(Y)\leq c_1(Y)\leq c_2(Y)\leq \ldots,\]
which are suitably selected elements of the action spectrum 
\begin{align}
\label{e:action_spectrum} 
\sigma(Y)=\big\{ \tau>0\ \big|\ \fix(\psi^\tau)\neq\varnothing  \big\}.
\end{align}
The first such value $c_0(Y)$ is the systole of $Y$, namely the minimum element of the action spectrum $\sigma(Y)$. We will recall the definition of the higher ones in Section~\ref{s:spectral_invariants}. A joint result of the first author with Ginzburg-G\"urel \cite{Ginzburg:2019ab} implies that these spectral invariants determine the Besse property: a $(2n-1)$-dimensional convex contact sphere $Y$ is Besse if and only if $c_i(Y)=c_{i+n-1}(Y)$ for some $i\geq0$.

Let us recall the computation of the Ekeland-Hofer spectral invariants for ellipsoids $E=E(a_1,\ldots,a_n)$. The action spectrum $\sigma(E)$ is the set of multiples of the parameters $a_i$. Let $\sigma_1<\sigma_2<\sigma_3<\ldots$ be the elements of $\sigma(E)$ listed in increasing order. For each $i\geq0$, we denote by $d_i$ the number of $a_j$'s such that $\sigma_i/a_j\in\N$. These numbers are indeed related to the dimensions of the subsets $\fix(\psi^{\sigma_i})\subset E$, i.e.
\begin{align*}
 d_i:=(\dim(\fix(\psi^{\sigma_i}))+1)/2.
\end{align*}
The Ekeland-Hofer spectral invariant $c_i(Y)$ is the $(i+1)$-th element in the sequence
\[
 \underbrace{\sigma_1,\ldots,\sigma_1}_{\times d_1},
 \underbrace{\sigma_2,\ldots,\sigma_2}_{\times d_2},
 \underbrace{\sigma_3,\ldots,\sigma_3}_{\times d_3},
  \ldots
\]
see, e.g., \cite[Section~3.7]{Ginzburg:2019ab}.
Our second theorem, shows that the computation of the Ekeland-Hofer spectral invariants of Besse convex contact spheres is the same as the one of ellipsoids.

\begin{maintheorem}\label{t:spec}
Let $Y$ be a Besse convex contact sphere with Reeb flow $\psi^t$. Let 
$\sigma_1<\sigma_2<\sigma_3<\ldots$ 
be the elements of the action spectrum $\sigma(Y)$ listed in increasing order. The Ekeland-Hofer spectral invariant $c_i(Y)$ is the $(i+1)$-th element in the sequence
\begin{align*}
 \underbrace{\sigma_1,\ldots,\sigma_1}_{\times d_1},
 \underbrace{\sigma_2,\ldots,\sigma_2}_{\times d_2},
 \underbrace{\sigma_3,\ldots,\sigma_3}_{\times d_3},
  \ldots
\end{align*}
where\footnote{The statement implicitly uses the fact that $\fix(\psi^{\sigma_j})$ is path-connected, which follows from Theorem~\ref{t:main}.} $d_j=(\dim(\fix(\psi^{\sigma_j}))+1)/2$.
\end{maintheorem}

We point out that in dimension 3, there are essentially no Besse contact spheres (convex or not) beside the rational ellipsoids. This is a consequence of existing results in the literature, which we summarize in the following statement. We recall that two contact manifolds $(Y_1,\alpha_1)$ and $(Y_2,\alpha_2)$ are strictly contactomorphic when there exists a diffeomorphism $\phi:Y_1\to Y_2$ such that $\phi^*\alpha_2=\alpha_1$. 

\begin{thm}[\cite{Geiges:2018aa, Cristofaro-Gardiner:2020aa}]
\label{t:ellipsoid}
Every Besse contact 3-sphere $(S^3,\alpha)$ is strictly contactomorphic to an ellipsoid. 
\end{thm}

\begin{proof}
Let $\tau>0$ be the minimal common Reeb period of the Besse $(S^3,\alpha)$. We denote by $\sigma_{\mathrm p}(S^3,\alpha)$ the simple action spectrum, which is the set of those $a>0$ that are minimal periods of some Reeb orbit of $(S^3,\alpha)$. The Reeb flow $\psi^t$ defines an $S^1=\R/\tau\Z$ action on $S^3$, and the quotient projection $S^3\to S^3/S^1$ is a Seifert fibration whose base is oriented by the 2-form $\pi_*\diff\alpha$ (this push-forward is well defined since $\psi^t$ preserves $\alpha$). The classification of Seifert fibrations \cite[Lemma~4.1]{Geiges:2018aa} implies that there are at most two distinct periodic Reeb orbits whose minimal period is strictly less than $\tau$. Therefore, $\sigma_{\mathrm p}(S^3,\alpha)$ contains at most three elements, one of which is $\tau$. If $\sigma_{\mathrm p}(S^3,\alpha)$ contains at least two elements, we call $a_1< a_2$ the two smallest elements of $\sigma_{\mathrm p}(S^3,\alpha)$; if instead $\sigma_{\mathrm p}(S^3,\alpha)=\{\tau\}$, we set $a_1=a_2=\tau$. In both cases, we conclude that $\sigma_{\mathrm p}(S^3,\alpha)$ is also the simple action spectrum of the Besse ellipsoid $E(a_1,a_2)$. Finally, by \cite[Theorem~1.5]{Cristofaro-Gardiner:2020aa}, two diffeomorphic Besse contact 3-manifolds are strictly contactomorphic if and only if they have the same prime action spectrum.
\end{proof}

In higher dimensions, Theorem~\ref{t:ellipsoid} does not hold: Ustilovsky \cite{Ustilovsky:1999vr} proved that, for each $n\geq3$ odd, there exist infinitely many Besse contact spheres $(S^{2n-1},\alpha_i)$, with $i\in\N$, that are pairwise not contactomorphic, meaning that for each pair of distinct positive integers $i_1\neq i_2$ there is no diffeomorphism $\phi:S^{2n-1}\to S^{2n-1}$ satisfying $d\phi(\ker\alpha_{i_1})=\ker\alpha_{i_2}$. Since all ellipsoids of dimension 
$2n-1$ are contactomorphic to one another, infinitely many Ustilovsky spheres $(S^{2n-1},\alpha_i)$ are not contactomorphic to ellipsoids, let alone strictly contactomorphic. 

A closed contact manifold all of whose Reeb orbits are closed and have the same minimal period is called \textbf{Zoll}. We remark that Ustilovsky's contact spheres are not Zoll. It is not known whether there exist two non strictly contactomorphic Zoll contact spheres of some dimension larger than or equal to~5. 
A celebrated related problem is the uniqueness (up to diffeomorphism) of the  symplectic form with unit volume on complex projective spaces, which is known for $\CP^2$ by a result of Taubes \cite{Taubes:1995aa}, but open in higher dimension. 

The proofs of the main results in this paper (and the very definition of the Ekeland-Hofer spectral invariants) are based on the Clarke action functional $\Psi$, which provides a variational principle for the periodic Reeb orbits of convex contact spheres. It is well known that, for Besse convex contact spheres, $\Psi$ is Morse-Bott non-degenerate. The technical core for the proof of Theorems~\ref{t:main} and  \ref{t:spec} is the following.

\begin{maintheorem}\label{t:perfect}
The Clarke action functional of a Besse convex contact sphere is perfect for the rational $S^1$-equivariant cohomology.
\end{maintheorem}

In this statement, perfect must be understood in the sense of Morse theory: for any $b>a>0$, the inclusions induce a short exact sequence
\begin{align*}
 0
 \to 
 H^*_{S^1}(\Psi^{-1}(0,b),\Psi^{-1}(0,a);\Q)
 \to 
 H^*_{S^1}(\Psi^{-1}(0,b);\Q)
 \to 
 H^*_{S^1}(\Psi^{-1}(0,a);\Q)
 \to 
 0.
\end{align*}
Namely, in the $S^1$-equivariant Morse theory of $\Psi$, there are no cancellations among the critical manifolds.

Theorem~\ref{t:perfect} will require two main ingredients, which turn out to always hold under the Besse assumption: the orientability of the negative normal bundle of any critical manifold (Proposition~\ref{p:orientability}), and the vanishing of the $S^1$-equivariant rational cohomology of any critical manifold in odd degree (Proposition~\ref{p:zero_odd_cohomology}). Both statements are inspired by analogous results for Besse geodesic flows due to the second author and Wilking \cite{Radeschi:2017dz}, which were employed to establish the high-dimensional Berger conjecture: every Besse geodesic flow on the unit cotangent bundle of a Riemannian $n$-sphere with $n>3$ is actually Zoll (the conjecture is still open for $n=3$, while it was earlier proved by Gromoll-Grove \cite{Gromoll:1981kl} for $n=2$).

Recent results in the literature \cite{Abbondandolo:2019aa, Irie:2019aa} seem to indicate that the Ekeland-Hofer spectral invariants of a convex contact sphere may coincide with its corresponding Ekeland-Hofer capacities \cite{Ekeland:1989aa, Ekeland:1990aa} (this is indeed known for the first one $c_0(Y)$, see \cite{Sikorav:1990aa}). Unlike the spectral invariants, the Ekeland-Hofer capacities can be associated to any compact subset of symplectic vector spaces; when associated to a restricted contact-type hypersurface $Y$ or to its compact filling, they take value inside its action spectrum $\sigma(Y)$. It is an open question whether Theorems~\ref{t:main} and~\ref{t:spec} (with the capacities instead of the spectral invariants) continue to hold for general restricted contact-type spheres in symplectic vector spaces.

\subsection{Organization of the paper}
In Section~\ref{s:Clarke}, we provide the background concerning the Clarke action functional. Section~\ref{s:orientability} is devoted to the proof of the orientability of the negative normal bundles of the critical manifolds of the Clarke action functional in the Besse case. In Section~\ref{s:perfectness}, we prove that the Clarke action functional of a Besse convex contact sphere is perfect for the $S^1$ equivariant rational cohomology, and use this fact to establish Theorem~\ref{t:perfect}, and then Theorems~\ref{t:main} and~\ref{t:spec}. Finally, in Appendix~\ref{a:Euler} we compute the Euler class of a suitable fibered sum of circle bundles, which is employed in Section~\ref{s:perfectness}.

\subsection*{Acknowledgments} The first author is grateful to Alberto Abbondandolo, Viktor Ginzburg, Ba\c{s}ak G\"urel, Christian Lange, and Jean-Claude Sikorav for several fruitful discussions, and to Jean Gutt for pointing out Ustilovsky's work  on the existence of exotic Besse contact spheres. The second author would like to thank Lee Kennard for discussions on groups actions on spheres. Both authors thank the anonymous referee for carefully reading the manuscript and for providing insightful comments.

\section{The Clarke action functional}\label{s:Clarke}

\subsection{The functional setting}

A convenient variational principle for the study of periodic Reeb orbits in a convex contact sphere is given by the Clarke action functional \cite{Clarke:1979aa}. Such functional appears in the literature under different, although equivalent, formulations; here we present the one in the $L^2$ setting, following Ekeland-Hofer \cite{Ekeland:1987aa}. 

Let $Y\subset\R^{2n}$ be a convex contact sphere. We embed its Reeb flow $\psi^t:Y\to Y$ into a Hamiltonian flow $\phi_H^t:\R^{2n}\to\R^{2n}$ so that $\phi_H^t(\lambda z)=\lambda\psi^t(z)$ for all $t\in\R$, $\lambda>0$, and $z\in Y$. We recall that $\phi_H^t$ is the flow of a Hamiltonian vector field $X_H$ on $\R^{2n}$ defined by $\omega(X_H,\cdot)=-\diff H$, where $H:\R^{2n}\to\R$ is the associated Hamiltonian, and $\omega=\diff\lambda$ is the standard symplectic 2-form of $\R^{2n}$. We denote by $J$ the standard complex structure of $\R^{2n}$, which is defined by $\langle J\cdot,\cdot\rangle=\omega$, and we remark that  $X_H=J\nabla H$. The Hamiltonian whose flow extends the Reeb flow as above is the unique function such that $H|_Y\equiv 1$ and $H(\lambda z)=\lambda^2$ for all $\lambda>0$ and $z\in Y$. 
In particular, $H$ is a convex 2-homogeneous function smooth outside the origin, and so is its dual
\begin{align*}
 H^*:\R^{2n}\to[0,\infty),\qquad H^*(w)=\max_{z\in\R^{2n}} \Big( \langle w,z\rangle - H(z)\Big).
\end{align*}

We consider the space of $L^2$ curves with zero average  
\[L^2_0(S^1,\R^{2n}) =\left\{ \zeta\in L^2(S^1,\R^{2n})\ \left|\ \int_{S^1} \zeta(t)\,\diff t=0 \right.\right\}, \]
where $S^1=\R/\Z$.
Notice that every element in $L^2_0(S^1,\R^{2n})$ is the first derivative $\dot\gamma$ of some element $\gamma\in W^{1,2}(S^1,\R^{2n})$.  
We consider the symplectic action functional
\begin{align*}
\A:L^{2}_0(S^1,\R^{2n})\to\R,\qquad
\A(\dot\gamma) = \int_{S^1} \gamma^*\lambda = \frac12 \int_{S^1} \langle J\gamma(t),\dot\gamma(t) \rangle\,\diff t.
\end{align*}
The expression of $\A$ involves a primitive $\gamma$, but the value $\A(\dot\gamma)$ is independent of its choice.

We next consider the functional
\begin{align*}
\G:L^2_0(S^1,\R^{2n})\to[0,\infty),
\qquad
\G(\dot\gamma) = \int_{S^1} H^*(-J\dot\gamma(t))\,\diff t.
\end{align*}
Notice that $\G(0)=0$, $\G$ is positive away from the origin, and $\G(\lambda\dot\gamma)=\lambda^2\G(\dot\gamma)$ for all $\lambda>0$ and $\dot\gamma\in L^2_0(S^1,\R^{2n})$. We set
\begin{align*}
\Lambda:=\G^{-1}(1)\cap\A^{-1}(0,\infty). 
\end{align*}
The \textbf{Clarke action functional} is defined by
\begin{align*}
 \Psi: \Lambda\to(0,\infty),\qquad \Psi(\dot\gamma)=\frac1{\A(\dot\gamma)}.
\end{align*}
The associated variational principle allows to characterize the closed Reeb orbits of $Y$ as follows.

\begin{thm}[Clarke variational principle]\label{t:Clarke}
The critical points \[\dot\gamma\in \crit(\Psi)\cap\Psi^{-1}(c)\] are precisely those curves admitting a (unique) primitive $\gamma\in C^\infty(S^1,\R^{2n})$ such that
$\dot\gamma=c X_H(\gamma)=cJ\nabla H(\gamma)$ and $H(\gamma)\equiv c^{-2}>0$.
Namely, the curve
$t\mapsto c\,\gamma(t/c)$ 
is a (possibly iterated) $c$-periodic Reeb orbit of $Y$.
\hfill\qed
\end{thm}

\subsection{Spectral invariants}\label{s:spectral_invariants}
The Clarke action functional $\Psi$ satisfies the Palais-Smale condition (see \cite{Ekeland:1987aa}).
Moreover, $\Psi$ has a remarkable feature that is  lacking in most of the other symplectic functionals: it is uniformly bounded from below. This allows to detect the closed Reeb orbits on $Y$ corresponding to the global minimizers of $\Psi$; the period of such Reeb orbits, sometimes called the systole of $Y$, is thus
\begin{align}
\label{e:systole}
 c_0(Y):=\min\Psi\in\sigma(Y),
\end{align}
where $\sigma(Y)$ is the action spectrum~\eqref{e:action_spectrum}

Consider the classifying space $BS^1$. Given a topological space $X$ equipped with a continuous $S^1$-action and a coefficient ring $R$, the $S^1$-equivariant cohomology 
\[H^*_{S^1}(X;R):=H^*(X\times_{S^1}ES^1;R)\] 
is a $H^*(BS^1;R)$-module with scalar multiplication 
\begin{align*}
f\cdot k:= (\pi^*f)\smallsmile k,\qquad \forall f\in H^*(BS^1;R),\ k\in H^*_{S^1}(X;R).
\end{align*}
Here, $\pi:X\times_{S^1}ES^1\to BS^1$ is the quotient-projection.
The Euler class of the universal principal $S^1$-bundle $ES^1\to BS^1$ induces a generator 
\[e\in H^2(BS^1;R).\] 
With a common abuse of notation, we still denote by $e$ the element $e\cdot 1\in H^2_{S^1}(X;R)$, which is the cohomology class induced by the Euler class of the principal $S^1$-bundle $X\times ES^1\to X\times_{S^1}ES^1$.

The Hilbert space $L^2_0(S^1,\R^{2n})$ is equipped with the $S^1$-action given by the time-translation 
\begin{align*}
t\cdot\dot\gamma=\dot\gamma(t+\cdot)\in\Lambda,\qquad\forall t\in S^1,\ \dot\gamma\in\Lambda.
\end{align*}
Notice that $\A$ and $\G$ are $S^1$-invariant, and so is the Clarke action functional $\Psi$. We denote its open sublevel sets by
\begin{align*}
 \Lambda^{<a}:=\Psi^{-1}(0,a),\qquad a\in(0,\infty].
\end{align*}
For each integer $i\geq0$, the $i$-th \textbf{Ekeland-Hofer spectral invariant} is defined by
\begin{align*}
c_i(Y) := \inf \big\{ c>0\ \big|\ e^i\neq0\mbox{ in }H^*_{S^1}(\Lambda^{<c};\Q) \big\}\in\sigma(Y).
\end{align*}
In the right-hand side, the non-vanishing of the cohomology class $e^i$ is often expressed in the literature by saying that $\Lambda^{<c}$ has Fadell-Rabinowitz index \cite{Fadell:1978aa} larger than or equal to $i$. The notation $c_i(Y)$ is consistent with the above one in Equation \eqref{e:systole}: $c_0(Y)$ is the systole of $Y$.

\subsection{The Morse index}\label{s:Morse_index}

Since the Hamiltonian $H$ and its dual $H^*$ are not $C^2$ at the origin, the Clarke action functional $\Psi$ is not $C^2$ neither. Nevertheless, it is $C^{1,1}$, and the finite dimensional reduction provided in \cite{Ekeland:1987aa} allows to treat it as a $C^2$ functional in the applications. 

Let us consider a critical point $\dot\gamma\in\crit(\Psi)\cap\Psi^{-1}(c)$. We set 
\[
\Phi:= \A|_{\Lambda}=\frac{1}{\Psi} .
\]
The Hessian bilinear form of the Clarke action functional at $\dot\gamma$ is given by
\begin{align*}
 \langle \nabla^2 \Psi(\dot\gamma)\dot\eta,\dot\zeta \rangle_{L^2}
 =
 -\frac{1}{\Phi(\dot\gamma)^2} \langle \nabla^2 \Phi(\dot\gamma)\dot\eta,\dot\zeta\rangle
 =
 -\frac{1}{\Phi(\dot\gamma)^2} 
 \int_{S^1} \langle J\eta,\dot\zeta\rangle,
\end{align*}
where $\dot\eta,\dot\zeta\in\Tan_{\dot\gamma}\Lambda$. The Hessian $\nabla^2 \Psi(\dot\gamma)$ is a self-adjoint bounded operator on $\Tan_{\dot\gamma}\Lambda$. By Theorem~\ref{t:Clarke}, $\dot\gamma$ admits a primitive $\gamma$ satisfying Hamilton's equation $\dot\gamma=cJ\nabla H(\gamma)$, and dually $c\,\gamma=\nabla H^*(-J\dot\gamma)$. Notice that $\Tan_{\dot\gamma}\Lambda=\Tan_{\dot\gamma}\G^{-1}(1)$. Since
\begin{align*}
\diff\G(\dot\gamma)\dot\eta
=\int_{S^1} \langle\nabla H^*(-J\dot\gamma),-J\dot\eta\rangle\,\diff t
=c\int_{S^1} \langle J\gamma,\dot\eta\rangle\,\diff t,\qquad\forall\dot\eta\in L^2_0(S^1,\R^{2n}),
\end{align*}
the tangent space $\Tan_{\dot\gamma}\Lambda$ is precisely the space of those $\dot\eta\in L^2_0(S^1,\R^{2n})$ such that
\begin{align*}
 \int_{S^1} \langle J\gamma,\dot\eta\rangle\,\diff t = 0.
\end{align*}

The \textbf{Morse index} $\ind(\Psi,\dot\gamma)$ is defined as the supremum of the dimension of those vector subspaces of $\Tan_{\dot\gamma}\Lambda$ over which $\nabla^2\Psi(\dot\gamma)$ is negative definite. The \textbf{nullity} $\nul(\Psi,\dot\gamma)$ is defined as the dimension of $\ker\nabla^2\Psi(\dot\gamma)$. Notice that $\ind(\Psi,\dot\gamma)=\ind(-\Phi,\dot\gamma)$ and  $\nul(\Psi,\dot\gamma)=\nul(-\Phi,\dot\gamma)=\nul(\Phi,\dot\gamma)$.

To compute the index of a critical point of $\Psi$, consider the functional
\begin{align*}
\Theta:L^2_0(S^1;\R^{2n})\to\R,
\qquad
\Theta(\dot\zeta)
=
 - \A(\dot\zeta) + c^{-1} \G(\dot\zeta).
\end{align*}
The critical points of $\Theta$ are precisely those $\dot\zeta$ admitting a primitive $\zeta$ such that 
\begin{align*}
 \dot\zeta = c J\nabla H(\zeta).
\end{align*}
Every critical point $\dot\zeta\in\crit(\Theta)$ belongs to a cylinder of critical points 
\[
(\R_+\times S^1)\cdot\dot\zeta
=
\big\{\rho\,\dot\zeta(t+\cdot)\ \big|\ \rho>0,\ t\in S^1\big\},
\]
and the only critical value of $\Theta$ is 0, because $\Theta$ is positively homogeneous of degree~2. Since
$\Theta|_{\Lambda} = c^{-1}-\Phi$,
we infer that our $\dot\gamma\in\crit(\Psi)\cap \Psi^{-1}(c)$ is also a critical point of $\Theta$. The Hessian of $\Theta$ at $\dot\gamma$ is given by
\begin{align*}
 \langle\nabla^2\Theta(\dot\gamma)\dot\eta,\dot\zeta\rangle
 & =
 - \int_{S^1} \langle J\eta,\dot\zeta\rangle + c^{-1}\int_{S^1} \langle\nabla^2H^*(-J\dot\gamma)J\dot\eta,J\dot\zeta\rangle\,\diff t\\
 & =
  \int_{S^1} \langle \eta + c^{-1}\nabla^2H(\gamma)^{-1}J\dot\eta,J\dot\zeta\rangle\,\diff t.
\end{align*}
This expression readily implies that $\dot\eta\in\ker(\nabla^2\Theta(\dot\gamma))$ if and only if $\dot\eta$ admits a primitive $\eta$ such that
$\dot\eta = c J\nabla^2 H(\gamma)\eta$.
Namely, 
\[\eta(t)=\diff\phi_H^{ct}(\gamma(0))\eta(0),\qquad\forall t\in [0,1].\] 
We recall that $\eta$ is a smooth 1-periodic curve, being the primitive of the zero-average curve $\dot\eta$. Therefore, the nullity of $\Theta$ at $\dot\gamma$ is given by
\begin{align*}
\nul(\Theta,\dot\gamma)=\dim\ker(\diff\phi_H^c(\gamma(0))-I).
\end{align*}
The computation of the Morse index is slightly more involved, but according to a theorem of Ekeland \cite[Theorem~I.4.6]{Ekeland:1990lc} turns out to be analogous to the one for geodesic arcs in Riemannian manifolds: it is given by a count of conjugate points. The formula is the following\footnote{Notice that the formula~\ref{e:Morse} does not directly correspond to the Morse index formula for closed geodesics in Riemannian manifolds. The index of a closed geodesic is indeed given by a count of conjugate points plus the so-called concavity.}
\begin{align}
\label{e:Morse}
 \ind(\Theta,\dot\gamma) & =\sum_{t\in(0,1)} \dim\ker(\diff\phi_H^{ct}(\gamma(0))-I). 
\end{align}
Since $\Tan_{\dot\gamma}\Lambda=\Tan_{\dot\gamma}\G^{-1}(1)$ and $\diff \G(\dot\gamma)\dot\gamma=2\G(\dot\gamma)=2$, the space $L^2(S^1,\R^{2n})$ splits as a direct sum
\begin{align*}
 L^2_0(S^1,\R^{2n})
 =
 T_{\dot\gamma}\Lambda
 \oplus
 \mathrm{span}\{\dot\gamma\},
\end{align*}
Notice that 
\[
\nabla^2\Theta|_{T_{\dot\gamma}\Lambda} = -\nabla^2\Phi.
\]
Since $\Theta$ is positively 2-homogeneous, we have 
\[\nabla^2\Theta(\dot\gamma)\dot\gamma=\nabla\Theta(\dot\gamma)=0.\]
The last two equations readily imply that
\begin{align*}
 \ind(\Psi,\dot\gamma) & = \ind(-\Phi,\dot\gamma) = \ind(\Theta,\dot\gamma) ,\\
 \nul(\Psi,\dot\gamma) & = \nul(-\Phi,\dot\gamma) = \nul(\Theta,\dot\gamma)-1 . 
\end{align*}
Finally, since $\dot\gamma$ belongs to a circle of critical points $S^1\cdot\dot\gamma$ of $\Psi$, we have
\begin{align*}
 \nabla^2\Psi(\dot\gamma)\ddot\gamma = \tfrac{\diff}{\diff t}\big|_{t=0}\nabla\Psi(t\cdot\dot\gamma)=0,
\end{align*}
and therefore the nullity $\nul(\Psi,\dot\gamma)$ is always at least 1.
From now on, we will simply write 
\[
\ind(\dot\gamma)=\ind(\Psi,\dot\gamma),
\qquad
\nul(\dot\gamma)=\nul(\Psi,\dot\gamma).
\]
Summing up, we have the following statement.

\begin{lem}
\label{l:Morse}
The Morse indices of a critical point $\dot\gamma\in\crit(\Psi)\cap\Psi^{-1}(c)$ are given by
\begin{align*}
 \ind(\dot\gamma) & =\sum_{t\in(0,c)} \dim\ker(\diff\phi_H^{t}(\gamma(0))-I),\\
 \nul(\dot\gamma) & =\dim\ker(\diff\phi_H^{c}(\gamma(0))-I)-1. \tag*{\qed}
\end{align*}
\end{lem}

Let us now assume that our convex contact sphere $Y$ is Besse with common Reeb period $\tau$. For any critical value $c$ of $\Psi$ we have $c/\tau= m/k$ for some relatively prime positive integers $m$ and $k$. The associated critical set $\crit(\Psi)\cap\Psi^{-1}(c)$ is diffeomorphic to the stratum $Y_k:=\fix(\psi^{\tau/k})$, which is a closed contact submanifold of $Y$.
The following statement summarizes the results in \cite{Ginzburg:2019ab} that are relevant here. We remark that the Morse indices in \cite{Ginzburg:2019ab} are referred to the Clarke action functional in the $L^p$ setting with $p\in(2,\infty)$, but turn out to coincide with the Morse indices in the $L^2$-setting $\Lambda$ employed here, see \cite[Prop.~I.7.5]{Ekeland:1990lc}. Indeed, for any $p\in(2,\infty)$, there is a correspondence between the critical points of the Clarke action functionals in $L^p$ and in $\Lambda$, and the Morse indices of corresponding critical points coincide.

\begin{lem}[Ginzburg-Gurel-Mazzucchelli \cite{Ginzburg:2019ab}]
\label{l:GGM}
Assume that $Y$ is a Besse convex contact sphere. Any path-connected component $K\subset\crit(\Psi)$ has even Morse index $\ind(K)$ and odd nullity $\nul(K)=\dim(K)$. In particular, $K$ is a non-degenerate critical manifold of $\Psi$ in the Morse-Bott sense. \hfill\qed
\end{lem}

Let $K_1,\ldots,K_r$ be the path-connected components of $\crit(\Psi)\cap\Psi^{-1}(c)$. We recall that the \textbf{negative bundle} of $K_i$ is the vector bundle $E^-_i\to K_i$ whose fiber at $\dot{\gamma}\in K_i$ is the direct sum of negative eigenspaces of $\nabla^2\Psi(\dot\gamma)$. 
Let $U^-_i\subset E^-_i$ be compact neighborhoods of the 0-sections, and $U^-=U^-_1\cup\ldots\cup U^-_r$ their disjoint union. If $\epsilon>0$ is small enough,  the sublevel set $\Lambda^{<c+\epsilon}$ is homotopy equivalent to $\Lambda^{<c-\epsilon}\cup_F U^-$, where $F:\partial U^-\to \Lambda^{c-\epsilon}$ is a suitable attaching map. If the negative bundles $E^-_i\to K_i$ are all orientable, then the excision and the Thom isomorphism imply
\[
h^*(\Lambda^{c+\epsilon},\Lambda^{c-\epsilon})\simeq h^*(U^-, U^-\setminus K)\simeq \bigoplus_i h^{*-\ind(K_i)}(K_i),
\]
where $h^*$ denotes either the $S^1$-equivariant or the ordinary cohomology with coefficients in a ring. This will be recalled with more details in the proof of Lemma~\ref{l:local_cohomology}.

\section{Orientability of the negative bundles}
\label{s:orientability}

The goal of this section is to prove that the negative bundles of the critical manifolds of the Clarke action functional $\Psi$ are all orientable.
\subsection{The index form}
We consider the symmetric bilinear form
\begin{gather*}
h_{A,\tau}:L^2_0([0,\tau],\R^{2n})\times L^2_0([0,\tau],\R^{2n})\to\R,
\\
h_{A,\tau}(\dot\zeta,\dot\eta) = \int_0^\tau \langle \zeta(t) + A(t)J\dot\zeta(t),J\dot\eta(t) \rangle\, \diff t,
\end{gather*}
where $\tau\in(0,1]$ and $A:[0,1]\to\Sym^+(2n)$ is a smooth path of positive-definite symmetric matrices. The associated bounded self-adjoint operator \[\HH_{A,\tau}\in\mathcal{L}(L^2_0([0,\tau],\R^{2n})),\] 
which satisfies $h_{A,\tau}(\dot\zeta,\dot\eta)=\langle \HH_{A,\tau}\dot\zeta,\dot\eta\rangle_{L^2}$, is given by
\begin{align*}
 \HH_{A,\tau}\dot\zeta= -J\zeta - JAJ\dot\zeta + \int_0^\tau \big( J\zeta + JAJ\dot\zeta \big)\,\diff t,
\end{align*}
where $\zeta$ is any primitive of $\dot\zeta$. One can show that the spectrum of $\HH_{A,\tau}$ consists of eigenvalues. For each $\lambda\leq0$, the eigenspace 
\[E_{A,\tau}(\lambda):=\ker(\HH_{A,\tau}-\lambda I)\] 
is the vector  space of those $\dot\zeta$ admitting a primitive $\zeta$ such that
\begin{align*}
 \left\{
   \begin{array}{@{}l}
    \dot\zeta = J(A-\lambda I)^{-1} \zeta \\ 
    \zeta(0)=\zeta(\tau)
  \end{array}
 \right.
\end{align*}
In particular, $E_{A,\tau}(\lambda)$ has dimension at most $2n$. It turns out that $\HH_{A,\tau}$ has only finitely many negative eigenvalues, and therefore has a finite dimensional total negative eigenspace
\begin{align*}
 E_{A,\tau}^-:=
 \bigoplus_{\lambda<0} E_{A,\tau}(\lambda).
\end{align*}
The Morse index $\ind(h_{A,\tau}):=\dim(E_{A,\tau}^-)$ is the largest dimension of a vector subspace over which $h_{A,\tau}$ is negative definite. If we denote the nullity by $\nul(h_{A,\tau}):=\dim(E_{A,\tau}(0))=\ker(\HH_{A,\tau})$, then $\ind(h_{A,\tau})+\nul(h_{A,\tau})$ is the largest dimension of a vector subspace over which $h_{A,\tau}$ is negative semidefinite. The inclusions
\begin{align*}
 L^2_0([0,\tau_1],\R^{2n})\hookrightarrow L^2_0([0,\tau_2],\R^{2n})\qquad \forall \tau_1<\tau_2\leq1
\end{align*}
readily imply that the functions $\tau\mapsto\ind(h_{A,\tau})$ and $\tau\mapsto\ind(h_{A,\tau})+\nul(h_{A,\tau})$ are monotone increasing. Actually, these functions are completely determined by the linear symplectic path $\Gamma_A:[0,\tau]\to\Sp(2n)$ defined by
\begin{align*}
 \dot\Gamma_A(t) & = JA(t)^{-1}\Gamma_A(t),\\
 \Gamma_A(0) & = I.
\end{align*}
Indeed, there are canonical isomorphisms
\begin{align}
\label{e:negative_eigenspace_h}
 E_{A,\tau}^- & \cong\bigoplus_{t\in(0,\tau)} \ker(\Gamma_A(t)-I),\\
 \nonumber
 E_{A,\tau}(0) & =\ker(\HH_{A,\tau})\cong \ker(\Gamma_A(\tau)-I).
\end{align}
We refer the reader to Ekeland's monograph \cite[Section~I.4]{Ekeland:1990lc} for more background concerning the indices of $h_{A}$. In Section~\ref{e:orientation},
we shall need the following fact.
\begin{lem}
\label{l:stability_eigenvalue_1}
Let $A_s:[0,1]\to\Sym^+(2n)$ be smooth paths, smoothly depending on $s\in\R$. Assume that $\ker(\Gamma_{A_0}(\tau)-I)\neq\{0\}$ for some $\tau\in(0,1)$. Then, for each $\epsilon>0$, there exists $\delta>0$ such that, for each $s\in(-\delta,\delta)$, we have $\ker(\Gamma_{A_s}(t)-I)\neq\{0\}$ for some $t\in[\tau-\epsilon,\tau+\epsilon]$.
\end{lem}

\begin{proof}
We set 
\begin{align*}
\iota(s,t) & :=\ind(h_{A_s,t}),\\
\nu(s,t) & :=\nul(h_{A_s,t})=\dim (\ker(\Gamma_{A_{s}}(t)-I)),\\ 
\kappa(s,t) & :=\iota(s,t)+\nu(s,t).
\end{align*}
so that, by our assumption, $\nu(0,\tau)>0$. By~\eqref{e:negative_eigenspace_h}, we have 
$\iota(0,\tau+\epsilon) > \kappa(0,\tau-\epsilon)$.
The function $\iota$ is lower semicontinuous, whereas the function $\kappa$ is upper semicontinuous. Therefore, for all $t\in[\tau-\epsilon,\tau+\epsilon]$ and for all $s$ sufficiently close to 0 we have
$\iota(0,t)\leq \iota(s,t)$ and $\kappa(0,t)\geq \kappa(s,t)$.
Let us assume by contradiction that there exists a sequence $s_n\to0$ such that $\nu(s_n,t)=0$ for all $t\in[\tau-\epsilon,\tau+\epsilon]$. By~\eqref{e:negative_eigenspace_h}, we have $\kappa(s_n,\tau-\epsilon)= \kappa(s_n,\tau+\epsilon) = \iota(s_n,\tau+\epsilon)$. But this implies
\begin{align*}
\kappa(0,\tau-\epsilon) \geq \kappa(s_n,\tau-\epsilon) = \iota(s_n,\tau+\epsilon) \geq \iota(0,\tau+\epsilon) > \kappa(0,\tau-\epsilon),
\end{align*}
which gives a contradiction.
\end{proof}

In the following, we will set $\tau=1$, and remove it from the notation:
\begin{align*}
h_{A}:=h_{A,1}, \qquad E_{A}(0):=E_{A,1}(0),\qquad E_{A}^-:=E_{A,1}^-.
\end{align*}

\begin{rem}
With the notation of Section~\ref{s:Morse_index}, the bilinear form $h_{A}$ is the Hessian of the functional $\Theta$ at a critical point $\dot\gamma$ when $A(t)=c^{-1}\nabla^2H(\dot\gamma(t))^{-1}$. In this case, the symplectic path $\Gamma_A$ is given by $\Gamma_A(t)=\diff\phi_H^{ct}(\gamma(0))$. The vector space $E_{A,\tau}^-$ is the negative eigenspace of the Hessian $\nabla^2\Theta(\dot\gamma)$, as well as the negative eigenspace of the Hessian of the Clarke action functional $\nabla^2\Psi(\dot\gamma)$.
\hfill\qed
\end{rem}

\subsection{A remark on the eigenvalues of positive symplectic paths}

We will need a variation of \cite[Prop.~I.3.1]{Ekeland:1990lc} from symplectic linear algebra.

\begin{lem}
\label{l:real_eigenvalues}
Let $A:\R\to\Sym^+(2n)$ and $\Gamma:\R\to\Sp(2n)$ be  smooth paths satisfying 
\[\dot\Gamma=JA\Gamma,\qquad\Gamma(0)^k=I\] 
for some positive integer $k$. Then  $\Gamma(t)$ has no real eigenvalues for all $t\neq0$ sufficiently close to $0$.
\end{lem}

\begin{proof}
Let us assume by contradiction that there exist sequences $t_n\to0$ and $\lambda_n\in\sigma(\Gamma(t_n))\cap\R$. Since the eigenvalues of $\Gamma(0)$ are $k$-th roots of unity, up to extracting a subsequence we must have $\lambda_n\to\lambda\in\{1,-1\}$. Since $\Gamma(0)$ is a $k$-th roots of the identity, its real eigenvalue $\lambda$ is semi-simple. If $2m=\dim\ker(\Gamma(0)-\lambda I)$, the eigenvalue $\lambda\in\sigma(\Gamma(0))$ branches into the $2m$ eigenvalues  $\mu_{1}(t),\ldots,\mu_{2m}(t)\in\sigma(\Gamma(t))$ for all $t$ sufficiently close to $0$, where every $\mu_i(t)$ depends continuously on $t$ and admits a left and right derivative at $t=0$ (see, e.g. \cite[Prop. 3.3]{Texier:2018aa}). Therefore, up to extracting a subsequence, we have a finite limit
\begin{align*}
 \lambda':= \lim_{n\to\infty} \frac{\lambda_n-\lambda}{t_n}\in\big\{\dot\mu_i(0^\pm)\ \big|\ i=1,\ldots,2m\big\}.
\end{align*}
Let $v_n\in\R^{2n}$ be an associated sequence of unit eigenvectors, i.e.\ $\Gamma(t_n)v_n=\lambda_nv_n$ and $\|v_n\|=1$. Once again, up to extracting a subsequence, we have $v_n\to v$, where $\|v\|=1$ and $\Gamma(0)v=\lambda v$. Notice that $\lambda=\lambda^{-1}$ and $\Gamma(0)^{-1}v=\lambda v$. Therefore
\begin{align*}
-\langle A(0)v,v \rangle
&=
\langle J\dot\Gamma(0)\Gamma(0)^{-1}v,v \rangle
=
\lambda\langle J\dot\Gamma(0)v,v \rangle\\
&=
\lambda\lim_{n\to\infty}
\left\langle J\frac{\Gamma(t_n)-\Gamma(0)}{t_n}v_n,v \right \rangle\\
&=
\lambda\lim_{n\to\infty}\left(
\frac{\lambda_n}{t_n}\langle J v_n , v \rangle - 
\frac{1}{t_n}\langle J \Gamma(0)v_n , v \rangle \right) \\
&=
\lambda\lim_{n\to\infty}\left(
\frac{\lambda_n}{t_n}\langle J v_n , v \rangle - 
\frac{1}{t_n}\langle J v_n , \Gamma(0)^{-1}v \rangle \right) \\
&=
\lambda\lim_{n\to\infty} \left(
\frac{\lambda_n}{t_n}\langle J v_n , v \rangle - 
\frac{\lambda}{t_n}\langle J v_n , v \rangle \right) \\
& = 
\lambda\lim_{n\to\infty} \frac{\lambda_n-\lambda}{t_n}\langle J v_n , v \rangle 
=
\lambda\lambda' \langle Jv,v\rangle=0.
\end{align*}
This contradicts the the fact that $A(0)$ is positive definite.
\end{proof}

\subsection{Orientation of the negative eigenspaces}
\label{e:orientation}

The following statement provides the orientability of the negative bundles of critical manifolds of the Clarke action functional of Besse convex contact spheres. An analogous statement for geodesic flows was proved by the second author and Wilking in \cite[Section~2]{Radeschi:2017dz}.

\begin{prop}
\label{p:orientability}
Let $k$ be a positive integer and $A:S^1\times [0,1]\to\Sym^+(2n)$ a smooth map, which we see as a loop of paths of symmetric positive definite matrices $A_s(t)=A(s,t)$, such that $s\mapsto\dim E^-_{A_s}$ is constant and $\Gamma_{A_s}(1)^k=I$ for all $s\in S^1$. The family of negative eigenspaces form a vector bundle
\begin{align}
\label{e:negative_bundle}
 \pi:E^-\to S^1,\qquad \pi^{-1}(s)=E^-_{A_s}
\end{align} 
that is trivial.
\end{prop}

\begin{proof}
We consider the spaces
\begin{align*}
G_{\leq 1} & = \big\{ (M,\lambda)\in\Sp(2n)\times(0,1]\ \big|\ \dim\ker(M-\lambda I)\leq1 \big\},\\
G_{1} & = \big\{ (M,\lambda)\in\Sp(2n)\times(0,1]\ \big|\ \dim\ker(M-\lambda I)=1 \big\},\\
\Sp_{\leq 1}(2n) & = \big\{ M\in\Sp(2n)\ \big|\ (M,\lambda)\in G_{\leq1} \quad\forall\lambda\in(0,1]\big\},
\end{align*}
equipped with their standard topology. The space $G_{\leq 1}$ is open in $\Sp(2n)\times(0,1]$. By~\cite[Appendix~A]{Radeschi:2017dz}, $G_1$ is a smooth hypersurface in $G_{\leq 1}$ with boundary
\begin{align*}
\partial G_1=(\Sp_{\leq 1}(2n)\times\{1\}) \cap G_1.
\end{align*}
Moreover, the complement $\Sp(2n)\setminus\Sp_{\leq 1}(2n)$ is a stratified manifold whose top-dimensional stratum has codimension~$3$.

We denote by $C=S^1\times[0,1]$ the cylinder, by $\mathrm{int}(C)=S^1\times(0,1)$ its interior, and set
\begin{align*}
 \Gamma_0:C\to\Sp(2n),\qquad \Gamma_0(s,t)=\Gamma_{A_s}(t).
\end{align*}
By the properties mentioned in the previous paragraph, we can smoothly extend $\Gamma_0$ to a homotopy
$\Gamma_r:C\to\Sp(2n)$, $r\in[0,1]$, such that:
\begin{itemize}
\item[(i)] each $\Gamma_r$ is $C^\infty$ close to $\Gamma_0$,
\item[(ii)] $\Gamma_r|_{\partial C}=\Gamma_0$ for all $r\in[0,1]$,
\item[(iii)] $\Gamma_1(\mathrm{int}(C))\subset \Sp_{\leq 1}(2n)$,
\item[(iv)] $\psi:\mathrm{int}(C)\times(0,1]\to G_{\leq 1}$, $\psi(s,t,\lambda)=(\Gamma_1(s,t),\lambda)$  is transverse to $G_1$.
\end{itemize}
Condition~(i) guarantees that each symmetric matrix 
\[B_r(s,t) = -J(\partial_t\Gamma_r(s,t))\Gamma_r(s,t)^{-1} \]
is positive definite (being close to $A_s(t)$). Condition~(ii) implies that the function 
\[r\mapsto\dim(E^-_{B_r(s,\cdot)})\]
is constant. Therefore, the vector bundle \eqref{e:negative_bundle} extends to a vector bundle
\begin{align*}
  \pi:\widetilde E^-\to [0,1]\times S^1,\qquad \pi^{-1}(r,s)=E^-_{B_r(s,\cdot)}.
\end{align*}
The vector bundles $\widetilde  E^-|_{\{0\}\times S^1}=E^-$ and $\widetilde E^-|_{\{1\}\times S^1}$ are isomorphic. Therefore, it is enough to prove the proposition for $B_1$ instead of $A$. In order to simplify the notation we will just assume that $A=B_1$, and thus $\Gamma:=\Gamma_1=\Gamma_0$.

Notice that $\Gamma|_{\partial C}$ takes values inside the subspace of symplectic roots of the identity $I$. Therefore, by Lemma~\ref{l:real_eigenvalues}, there exists an open neighborhood $U\subset C$ of $\partial C$ such that, for each $u=(s,t)\in U\setminus \partial C$,
the matrix $\Gamma(u)$ has no real eigenvalues. This, together with the transversality condition (iv), implies that the preimage $W:=\psi^{-1}(G_1)$ is a compact surface embedded in $C\times(0,1]$ with boundary $\partial W=W\cap(C\times\{1\})$.

We denote the Morse index of $h_{A_s}$ by
\[i:=\dim E^-_{A_s},\] 
which is independent of $s\in S^1$ by assumption. Since $\Gamma(\mathrm{int}(C))\subset\Sp_{\leq1}(2n)$, for each $s\in S^1$ there exist $0<\tau_1(s)<\ldots<\tau_i(s)<1$ such that \[\ker(\Gamma(s,\tau_j(s))-I)\neq\{0\},\qquad \forall j=1,\ldots,i.\] 
Namely, for each $s\in S^1$, the intersection $\partial W\cap(\{s\}\times[0,1]\times \{1\})$ has precisely $i$ elements given by
\begin{align*}
  W\cap(\{s\}\times[0,1]\times \{1\}) = \big\{(s,\tau_j(s),1)\ \big|\ j=1,\ldots,i\big\}.
\end{align*}
By Lemma~\ref{l:stability_eigenvalue_1}, the functions $s\mapsto\tau_j(s)$ are continuous. Therefore, the boundary $\partial W=W\cap(C\times\{1\})$ is the disjoint union of $i$ embedded circles $T_1\cup\ldots\cup T_i$, where 
\[T_j=\big\{(s,\tau_j(s),1)\ \big|\ s\in S^1\big\}, \qquad j=1,\ldots,i.\]
Since the circles $T_j$ are smooth, the functions $s\mapsto\tau_j(s)$ are smooth as well.
Notice that, since $\tau_j(s)<\tau_{j+1}(s)$, every circle $T_j$ is enclosed by the subsequent one $T_{j+1}$.

Now, we consider the real line bundle 
\begin{align*}
\pi:R\to W,\qquad \pi^{-1}(s,t,\lambda) = \ker(\Gamma(s,t)-\lambda I).
\end{align*}
For each $j=1,\ldots,i$, we define $r_j\in\{0,1\}$ as
\begin{align*}
 r_j
 :=
 \left\{
   \begin{array}{@{}ll}
    0,   & \mbox{if $R|_{T_j}\to T_j$ is orientable,} \vspace{3pt} \\ 
    1,   & \mbox{otherwise.}  
  \end{array}
 \right.
\end{align*}
Since $W$ is a compact surface embedded in $C\times[0,1]$ with boundary $\partial W=T_1\cup\ldots\cup T_i\subset C\times\{1\}$, it must be orientable, and therefore $r_1+\ldots+r_i$ is even. For each $j=1,\ldots,i$, we introduce the diffeomorphism
\begin{align*}
 \theta_j:S^1\ttoup^{\cong} T_j,\qquad\theta_j(s)=(s,\tau_j(s),1),
\end{align*}
and consider the vector bundle
\begin{align*}
N:=\theta_1^*R\oplus\ldots\oplus\theta_k^*R \to S^1.
\end{align*}
Notice that every fiber of this bundle is isomorphic to the corresponding fiber of the negative bundle $E^-\to S^1$ of Equation~\eqref{e:negative_bundle} via the canonical isomorphism~\eqref{e:negative_eigenspace_h}. Therefore, $N\to S^1$ and $E^-\to S^1$ are isomorphic vector bundles. Finally, $N$ is the direct sum of real line bundles, $r_1+\ldots+r_i$ of which are non-orientable. Since $r_1+\ldots+r_i$ is even, $N$ is orientable, and thus a trivial vector bundle.
\end{proof}

\subsection{Critical sets of Besse convex contact spheres}

As we mentioned, Proposition~\ref{p:orientability} guarantees that the negative bundles of the critical manifolds of the Clark action functional are always orientable under the Besse assumption. A standard argument from Morse theory thus implies that the critical manifolds are all homologically visible. More precisely, we have the following statement. Given a topological space $X$, we denote as usual by $\pi_0(X)$ the family of its path-connected components.

\begin{lem}
\label{l:local_cohomology}
Let $Y$ be a Besse convex contact sphere. We denote by $h^*$ either the $S^1$-equivariant or the ordinary cohomology functor with coefficients in a ring $R$.
\begin{itemize}
\item[$\mathrm{(i)}$] If an interval $[a,b)\subset\R$ contains a unique critical value $c$ of $\Psi$, then the critical set $K_c=\crit(\Psi)\cap\Psi^{-1}(c)$ has local cohomology
\begin{align*}
h^*(\Lambda^{<b},\Lambda^{<a}) \cong \bigoplus_{K\in\pi_0(K_c)} \!\!\!  h^{*-\ind(K)}(K).
\end{align*}
If $h^*$ is the $S^1$-equivariant cohomology with coefficients in $R$, this is an isomorphism of $H^*(BS^1;R)$-modules.

\item[$\mathrm{(ii)}$] Let $[a,b)\subset\R$ be an interval, and $d\geq0$ an integer. If
\begin{align*}
h^{d-\ind(K)}(K)=0,\qquad\forall K\in\pi_0(\crit(\Psi)\cap\Psi^{-1}[a,b)),
\end{align*}
then $h^d(\Lambda^{<b},\Lambda^{<a}) =0$.

\end{itemize}
\end{lem}

\begin{proof}
Even though the Clarke action functional $\Psi$ may only be $C^{1,1}$ for our Besse convex contact sphere $Y$, the finite dimensional reduction provided in \cite{Ekeland:1987aa} allows to treat it as a $C^2$ functional in the applications. The Besse assumption implies that $\Psi$ is Morse-Bott (Lemma~\ref{l:GGM}). Any connected component $K\subset\crit(\Psi)$ is a closed manifold. We consider the negative bundle $\pi:E^-\to K$, which is the vector bundle of rank $\ind(K)$, whose fibers $\pi^{-1}(\dot\gamma)$ are the negative eigenspaces of the Hessian $\nabla^2\Psi(\dot\gamma)$. We also consider the positive bundle $\pi:E^+\to K$, whose fibers $\pi^{-1}(\dot\gamma)$ are the infinite-dimensional positive eigenspaces of the Hessian $\nabla^2\Psi(\dot\gamma)$. The direct sum $E^+\oplus E^-\to K$ is isomorphic to the normal bundle $NK\to K$. Given sufficiently small $S^1$-invariant open neighborhoods $U^\pm\subset E^\pm$ of the 0-section, the Morse Lemma allows to identify  $U^+\oplus U^-\subset E^+\oplus E^-$ with an $S^1$-invariant tubular neighborhood $U\subset \Lambda$ of $K$;  under such an identification, the 0-section of $E^+\oplus E^-$ corresponds to the critical set $K$, and the Clarke action functional takes the form
\begin{align*}
\Psi(x,y)= c -\|x\|_{L^2}^2 + \|y\|^2_{L^2},\qquad\forall (x,y)\in U^-\oplus U^+\equiv U,
\end{align*}
where $c=\Psi(K)$ is the critical value. We denote by $U^{<c}:=\{u\in U\ |\ \Psi(u)<c\}$ the critical sublevel set of $\Psi|_U$. Since the functional $\Psi|_U$ is a non-degenerate quadratic form in the normal directions to $K$, the inclusion $(U^-,U^-\setminus\{0\})\hookrightarrow (U,U^{<c})$ admits an  $S^1$-equivariant homotopy inverse, and in particular induces an isomorphism
\begin{align}
\label{e:iso_U_1}
 h^*(U,U^{<c}) \toup^{\cong} h^*(U^-,U^-\setminus\{0\}).
\end{align}
By Proposition~\ref{p:orientability}, the negative bundle $E^-\to K$ is orientable. Therefore, the cup product with the Thom class $\tau\in h^{\ind(K)}(U^-,U^-\setminus\{0\})$ induces an isomorphism
\begin{align}
\label{e:iso_U_2}
 h^{*-\ind(K)}(K) \ttoup^{\smallsmile \tau}_{\cong} h^*(U^-,U^-\setminus\{0\}).
\end{align}
Notice that, when $h^*$ is the $S^1$-equivariant cohomology with coefficients in a ring $R$, both \eqref{e:iso_U_1} and \eqref{e:iso_U_2} are isomorphisms of $H^*(BS^1;R)$-modules.

Assume now that the interval $[a,b)$ contains a unique critical value $c$ of $\Psi$. Since $\Psi$ satisfies the Palais-Smale condition, the critical set $\crit(\Psi)\cap\Psi^{-1}(c)$ has finitely many path-connected components $K_1,\ldots,K_r$. Let $U_i\subset\Lambda^{<b}$ be an $S^1$-invariant tubular neighborhood of $K_i$ given by the Morse lemma. We require the $U_i$'s to be small enough so that they are pairwise disjoint, and we set $U:=U_1\cup\ldots\cup U_r$. The inclusions 
$\Lambda^{<c}\cup U \hookrightarrow \Lambda^{<b}$
and
$\Lambda^{<a} \hookrightarrow \Lambda^{<c}$
admit $S^1$-equivariant homotopy inverses that can be constructed by pushing with the anti-gradient flow of $\Psi$. In particular, we have the isomorphisms induced by the inclusion
\begin{equation*}
 \begin{tikzcd}
 h^*(\Lambda^{<b},\Lambda^{<a}) 
 &
 h^*(\Lambda^{<b},\Lambda^{<c}) 
 \arrow[r, "\cong"]  \arrow[l, "\cong"'] 
 &
 h^*(\Lambda^{<c}\cup U,\Lambda^{<c}) .
\end{tikzcd}
\end{equation*}
The excision property of cohomology, and the isomorphisms~\eqref{e:iso_U_1} and~\eqref{e:iso_U_2} give
\begin{equation*}
 \begin{tikzcd}
 h^*(\Lambda^{<c}\cup U,\Lambda^{<c}) 
 \arrow[r, "\cong"]
 &
 \displaystyle\bigoplus_{i=1}^r h^*(U_i,U_i^{<c})
 \arrow[r, "\cong"]
 &
 \displaystyle\bigoplus_{i=1}^r h^{*-\ind(K_i)}(K_i).
\end{tikzcd}
\end{equation*}
When $h^*$ is the $S^1$-equivariant cohomology with coefficients in a ring $R$, all the arrows are isomorphisms of $H^*(BS^1;R)$-modules. This proves point (i).

As for point (ii), we fix a sequence of real numbers $a=a_0<a_1<a_2< \ldots \leq b$ such that each interval $[a_j,a_{j+1})$ contains at most one critical value of $\Psi$. If $b<\infty$, we require this sequence to contain finitely many elements, say $k+1$, and $a_k=b$; Otherwise, the sequence is an infinite one and  $a_j\to\infty$. Assume now that $h^{d-\ind(K)}(K)=0$ for each connected component $K\in\pi_0(\crit(\Psi)\cap\Psi^{-1}[a,b))$, so that
$h^d(\Lambda^{<a_{j+1}},\Lambda^{<a_j})=0$ according to point (i) of the lemma. For each $j_1<j_2<j_3$, we have a long exact sequence
\begin{align*}
 \ldots\toup
 h^*(\Lambda^{<a_{j_3}},\Lambda^{<a_{j_2}})
 \toup
 h^*(\Lambda^{<a_{j_3}},\Lambda^{<a_{j_1}})
 \toup
 h^*(\Lambda^{<a_{j_2}},\Lambda^{<a_{j_1}})
 \toup
 \ldots
\end{align*}
which, together with the vanishing of  $h^d(\Lambda^{<a_{j+1}},\Lambda^{<a_j})=0$, readily implies that $h^d(\Lambda^{<a_{j_2}},\Lambda^{<a_{j_1}})=0$ for all $j_1<j_2$. This proves point (ii) when $b$ is finite. When $b=\infty$, the statement follows by taking the inverse limit of $h^d(\Lambda^{<a_j},\Lambda^{<a_0})$.
\end{proof}

\section{Perfectness of the Clarke action functional}
\label{s:perfectness}

\subsection{Topology of the domain of the Clarke action functional}
One of the main ingredients for the proof of Theorem~\ref{t:perfect} is the fact that the Clarke action functional $\Psi:\Lambda\to(0,\infty)$ of a Besse convex contact sphere is perfect in the sense of Morse theory for the $S^1$-equivariant rational cohomology. In order to prove this fact, we first need the following statement concerning the topology of $\Lambda$, which is certainly well known to the experts.

\begin{lem}
\label{l:domain_Clarke}
The space $\Lambda$ is $S^1$-equivariantly homotopy equivalent to the unit sphere of a separable complex Hilbert space. In particular, $\Lambda$ is contractible and its $S^1$-equivariant cohomology ring is given by
\begin{align*}
 H^*_{S^1}(\Lambda;\Z)
 =
 \Z[e],
\end{align*}
where $e\in H^2_{S^1}(\Lambda;\Z)$ is the Euler class of the $S^1$-bundle $\Lambda\times ES^1\to\Lambda\times_{S^1}ES^1$.
\end{lem}

\begin{proof}
The Hilbert space $L^2_0(S^1,\R^{2n})$ splits as a orthogonal direct sum decomposition $E_-\oplus E_+$, where
\begin{align*}
 E_\pm & =  \bigoplus_{\pm k>0} E_k,\\
 E_k & = \mathrm{span}\big\{t\mapsto \exp(2\pi k t J )v\ \big|\ v\in\R^{2n}\big\}.
\end{align*}
Notice that the $S^1$-action on $L^2_0(S^1,\R^{2n})$ is a diagonal action on each factor $E_k$. The functional $\A$ is a non-degenerate quadratic form of zero signature on $L^2_0(S^1,\R^{2n})$, for
\begin{align*}
 \A(\dot\gamma) = \sum_{k\neq 0} \frac{1}{2\pi k} \|\dot\gamma_k\|_{L^2}^2.
\end{align*}
Here, we have written $\dot\gamma$ according to the above orthogonal direct sum decomposition as
\begin{align*}
 \dot\gamma=\sum_{k\neq 0} \dot\gamma_k,
\end{align*}
with $\dot\gamma_k\in E_k$. The inclusion 
$\Lambda\hookrightarrow\A^{-1}(0,\infty)$
is an $S^1$-equivariant homotopy equivalence, whose homotopy inverse is the time-1 map $r_1:\A^{-1}(0,\infty)\to \Lambda$ of the $S^1$-equivariant deformation retraction
\begin{align*}
r_s:\A^{-1}(0,\infty)\to \A^{-1}(0,\infty),
\qquad
r_s(\dot\gamma)= (1+s(\sqrt{\G(\dot\gamma)}-1))^{-1} \dot\gamma.
\end{align*}
Clearly, the inclusion $E_+\setminus\{0\}\hookrightarrow \A^{-1}(0,\infty)$ is an $S^1$-equivariant homotopy equivalence as well, and so is the inclusion of the unit sphere $S(E_+)\hookrightarrow E_+\setminus\{0\}$. Summing up, we have an $S^1$-equivariant homotopy equivalence
$f:\Lambda\to S(E_+)$,
which induces an isomorphism 
\[f^*:H^*_{S^1}(S(E_+);\Z)\toup^{\cong}H^*_{S^1}(\Lambda).\]
The infinite dimensional sphere $S(E_+)$ is contractible, and therefore the Gysin sequence
\begin{align*}
  H^{d+1}(S(E_+);\Z) \toup H^d_{S^1}(S(E_+);\Z) \ttoup^{\smallsmile e} H^{d+2}_{S^1}(S(E_+);\Z) \toup H^{d+2}(S(E_+);\Z) 
\end{align*}
implies $H^*_{S^1}(S(E_+);\Z)=\Z[e]$. Here $e\in H^2_{S^1}(S(E_+);\Z)$ is the Euler class of the principal $S^1$-bundle $S(E_+)\times ES^1\to S(E_+)\times_{S^1}ES^1$. Finally, by the naturality of the Euler class, $f^*(e)$ is the Euler class of $\Lambda\times ES^1\to\Lambda\times_{S^1}ES^1$.
\end{proof}

\subsection{Torsion of the cohomology of iterated $S^1$-spaces}
In the proof of the perfectness of $\Psi$, surprisingly a crucial role is played by the $p$-torsion of the $S^1$-equivariant cohomology of the sublevel sets of $\Psi$ for a large enough prime $p$. Such a role is governed by the following proposition, which was proved by Radeschi-Wilking. Even though in the original source \cite[Prop.~5.3, 5.8]{Radeschi:2017dz} the statement is  phrased for the manifolds of closed geodesics of a Besse Riemannian manifold, the proof goes through in a general abstract setting, and we provide full details for the reader's convenience. In the statement and later on, we will adopt the following notation: if $X$ is a space equipped with an $S^1$-action $(t,x)\mapsto t\cdot x$, for each $q\in \N$ we denote by $X^q$ the same space equipped with the $S^1$-action $(t,x)\mapsto qt\cdot x$. Notice that the $S^1$-equivariant rational cohomologies of $X$ and $X^q$ are isomorphic, for
\[
 H^*_{S^1}(X;\Q)
 \cong
 H^*(X/S^1;\Q)
 \cong
 H^*(X^q/S^1;\Q)
 \cong
 H^*_{S^1}(X^q;\Q).
\]
However, with integer coefficients, the $S^1$-equivariant cohomologies can be different.

\begin{prop}[Radeschi-Wilking \cite{Radeschi:2017dz}]
\label{p:torsion}
Let $p$ be a prime number not dividing the order of any torsion element of $H^*_{S^1}(X;\Z)$. Then:
\begin{itemize}
\item[$\mathrm{(i)}$] If $H^{d}_{S^1}(X^q;\Z)$ has non-trivial $p$-torsion for some $d$ and $q$, then $q/p\in\N$ and $H^{d-2m}_{S^1}(X;\Q)\neq0$ for some  $m\geq0$.

\item[$\mathrm{(ii)}$] If $H^{d}_{S^1}(X;\Q)\neq 0$ and $H^{d+2}_{S^1}(X;\Q)=0$ for some $d\geq0$, then $H^{d+2m}_{S^1}(X^p;\Z)$ has non-trivial $p$-torsion for every $m>0$.
\end{itemize}
\end{prop}

\begin{proof}
We fix, once for all, a prime number $p$ not dividing the order of any torsion element in $H^*_{S^1}(X;\Z)$.
Since we are only interested in the $p$-torsion, we shall consider  singular cohomology groups with coefficients in the ring
\begin{align*}
R:=\big\{\tfrac ab  \big|\ a,b\in\Z,\ \tfrac bp\not\in\Z\big\}\subset\Q.
\end{align*}
Notice that $H^*_{S^1}(X;R)$ is torsion-free.

We employ a trick that will allow us to relate the $S^1$-equivariant cohomology of $X^q$ with the one of $X$ by means of a suitable Gysin sequence. We consider the space
\[ Y:=\frac{(X\times ES^1)^q \times ES^1}{S^1},\]
where $S^1$ is understood to act on $(X\times ES^1)^q \times ES^1$ by 
\[t\cdot(x,v_1,v_2)=(qt\cdot x,qt\cdot v_1,t\cdot v_2).\]
The $S^1$-equivariant projection map $Y\to X^q\times_{S^1} ES^1$, $[x,v_1,v_2]\mapsto[x,v_2]$ is a fibration with contractible fibers homeomorphic to $ES^1$. In particular, it is an $S^1$-equivariant homotopy equivalence, and induces an isomorphism
\begin{align}
\label{e:iso_X^q_Y}
 H_{S^1}^*(X^q;R)\cong H^*(Y;R).
\end{align}
We now consider the space
\begin{align*}
 Z:=\frac{(X\times ES^1)^q \times ES^1}{S^1\times S^1}=\frac{X\times ES^1}{S^1}\times\frac{ES^1}{S^1}=(X\times_{S^1} ES^1)\times BS^1,
\end{align*}
where $S^1\times S^1$ here acts on $(X\times ES^1)^q\times ES^1$ by
\[(t_1,t_2)\cdot(x,v_1,v_2)=(qt_1\cdot x,qt_1\cdot v_1,t_2\cdot v_2).\]
Notice that $Z$ is independent of $q$. If we equip $Y$ with the $S^1$ action
\begin{align*}
 t\cdot[x,v_1,v_2]=[qt\cdot x,qt\cdot v_1,v_2]=[x,v_1,-t\cdot v_2],
 \qquad\forall t\in S^1,\ [x,v_1,v_2]\in Y,
\end{align*}
the projection map $\pi:Y\to Z$ is a principal $S^1$ bundle. By Lemma~\ref{l:Euler}, its Euler class $e\in H^2(Z;R)$ is given by
\begin{align*}
e = e_1\otimes 1 - q\cdot\,1\otimes e_2,
\end{align*}
where $e_1\in H^2_{S^1}(X;R)$ and $e_2\in H^2(BS^1;R)$ are the Euler classes of $X\times ES^1\to X\times_{S^1} ES^1$ and $ES^1\to BS^1$ respectively.
By~\eqref{e:iso_X^q_Y}, the Gysin long exact sequence of $\pi:Y\to Z$ reads
\begin{align*}
 \ldots
 \ttoup^{\pi_*}
 H^{*-2}(Z;R)
 \ttoup^{e\,\smallsmile} 
 H^*(Z;R)
 \ttoup^{\pi^*}
 H^*_{S^1}(X^q;R)
 \ttoup^{\pi_*}
 H^{*-1}(Z;R)
 \ttoup^{e\,\smallsmile}\ldots
\end{align*}
We recall that $H^*_{S^1}(BS^1;R)=R[e_2]$, and in particular it vanishes in odd degrees and is isomorphic to $R$ in every even degree. By the K\"unneth formula, we have
\begin{align*}
 H^*(Z;R) \cong H^*_{S^1}(X;R)\otimes H^*(BS^1;R),
\end{align*}
that is, every cohomology class in $H^d(Z;R)$ is a linear combination of terms of the form $k\otimes e_2^i$, where $0\leq i\leq\lfloor d/2\rfloor$ and $k\in H^{d-2i}_{S^1}(X;R)$.

We can now prove point (i). Let us assume that there exists a non-zero cohomology class $k\in H^{d}_{S^1}(X^q;R)$ such that $pk=0$. Since $H^*(Z;R)$ is torsion-free, we must have $\pi_*k=0$. The above Gysin sequence implies that $k=\pi^* k'$ for some non-zero $k'\in H^{d}(Z;R)$. Since $H^*(Z;R)$ is torsion-free, this proves that the rational cohomology $H^{d}(Z;\Q)$ is non-trivial, and therefore that $H^{d-2m}_{S^1}(X;\Q)$ is non-trivial as well for some $m\geq0$. It remains to show that $p$ divides $q$. 

Let us assume by contradiction that $q/p\not\in\N$, so that we can always divide by $q$ in the ring $R$. Since $\pi^* pk'=0$, the above Gysin sequence implies that $pk'=e\smallsmile k''$ for some $k''\in H^{d-2}(Z;R)$. The cohomology classes $k'$ and $k''$ can be uniquely written as  
\[
k'=\sum_{i=0}^{\lfloor d/2\rfloor} k_i'\otimes e_2^{i},
\qquad
k''=  \!\!\!\sum_{i=0}^{\lfloor(d-2)/2\rfloor}\!\!\! k_i''\otimes e_2^{i},
\] 
where $k_i'\in H^{d-2i}_{S^1}(X;R)$ and $k_i''\in H^{d-2-2i}_{S^1}(X;R)$. Let $r\leq \lfloor(d-2)/2\rfloor$ be the largest integer such that $k_r''\neq0$.
The identity $pk'=e\smallsmile k''$ can be rewritten as 
\begin{align*}
p k_{r+1}' & = - q\,k_{r}'',\\ 
p k_{r}' & = e_1\smallsmile k_{r}''- q\,k_{r-1}'',\\ 
\vdots &\\
p k_{1}' & = e_1\smallsmile k_{1}'' - q\,k_{0}'',\\ 
p k_{0}' & = e_1\smallsmile k_{0}''. 
\end{align*}
This readily implies that $k_r''=-p\,q^{-1}k_{r+1}'$. Let us now prove by induction that every $k_i''$ is divisible by $p$: if this holds for $k_r'', k_{r-1}'',\ldots,k_{i+1}''$, the identity  \[k_i''=-p\, q^{-1} k_{i+1}' + q^{-1} e_1\smallsmile k_{i+1}''\] implies that it holds for $k_i''$ as well. Therefore  $k''=pk'''$ for some $k'''\in H_{S^1}^{d-2}(Z;R)$, and since $pk'=e\smallsmile pk'''=p(e\smallsmile k''')$, we have $k'=e\smallsmile k'''$. However, the above Gysin sequence gives the contradiction
\begin{align*}
0\neq k = \pi^*k'= \pi^*(e\smallsmile k''') = 0.
\end{align*}

We now set $q=p$, and assume that $H^{d}_{S^1}(X;\Q)\neq 0$ and $H^{d+2}_{S^1}(X;\Q)=0$. By our choice of $p$, the same holds if we employ the singular cohomology with coefficients in $R$. Therefore, we can find a non-zero $w\in H^{d}_{S^1}(X;R)$ of infinite order, not divisible by $p$, and such that $e_1\smallsmile w=0$. For all integers $m\geq 1$, we have
\begin{align*}
 e\smallsmile(w\otimes e_2^{m-1}) = -p\,w\otimes e_2^m \neq 0.
\end{align*}
We claim that $w\otimes e_2^m$ is not in the image of the map $y\mapsto e\smallsmile y$ that appears in the above Gysin sequence. Indeed, if \[y=y_0\otimes 1+y_1\otimes e_2+y_2\otimes e_2^2+\ldots+y_h\otimes e_2^h\] with $y_h\neq0$, the cohomology class $e\smallsmile y$ is the sum of $-p y_h\otimes e_2^{h+1}$ and of other terms of the form $z_i\otimes e_2^i$ with $i\leq h$. If $w\otimes e_2^m=e\smallsmile y$, then $h+1= m$ and $w=p\, y_{m-1}$, contradicting the fact that $w$ is not divisible by $p$. The above Gysin sequence implies that 
\[
\pi^*(w\otimes e_2^m)\neq 0,\qquad p\,\pi^*(w\otimes e_2^m)
=
\pi^*(p\, w\otimes e_2^m)
=
-\pi^*(e\smallsmile(w\otimes e_2^{m-1}))
=
0.
\] 
In particular $H^{d+2m}_{S^1}(X^p;R)$ has $p$-torsion for all $m\geq1$, and so does $H^{d+2m}_{S^1}(X^p;\Z)$.
\end{proof}

\subsection{Cohomology of the critical sets}

For each $\dot\gamma\in\Lambda$ and $m\in\N$, we denote by $\dot\gamma^m\in\Lambda$ the $m$-th iterate of $\dot\gamma$, which is defined by \[\dot\gamma^m(t)=\dot\gamma(mt).\] Notice that, according to the Clarke variational principle, $\dot\gamma\in\crit(\Psi)$ if and only if $\dot\gamma^m\in\crit(\Psi)$ for all $m\geq1$. 

We assume that our convex contact sphere $Y\subset\R^{2n}$ is Besse. Let $\tau>0$ be the minimal common period of its closed Reeb orbits. There are finitely many integers $1=k_s<k_{s-1}<\ldots<k_1$ such that each quotient $\tau/k_h$ is the (not necessarily minimal) period of some Reeb orbit. We set 
\begin{align*}
 P_h:=\crit(\Psi)\cap\Psi^{-1}(\tau/k_h),\qquad h=1,\ldots,s.
\end{align*}
Theorem~\ref{t:main} will imply that each $P_h$ is a path-connected component of $\crit(\Psi)$, but for now we only know that it is a finite union of such path-connected components. If we denote by $\psi^t$ the Reeb flow of $Y$ and by $Y_{k_h}:=\fix(\psi^{\tau/k_h})$ the $k_h$-stratum of $Y$, there is an equivariant diffeomorphism
\begin{align}
\label{e:equiv_diff_crit_strat}
P_h\toup^{\cong} Y_{k_h},\qquad \dot\gamma\mapsto \tfrac{\tau}{k_h}\gamma(0).
\end{align}
Here, $\gamma$ is the unique primitive of $\dot\gamma$ such that $t\mapsto \tfrac{\tau}{k_h}\gamma\big(t\tfrac{k_h}{\tau}\big)$ is a closed Reeb orbit of $Y$ (see Theorem~\ref{t:Clarke}). The equivariance in~\eqref{e:equiv_diff_crit_strat} intertwines the $S^1$ action on $P_h$ and the $\R/(\tau/k_h)\Z$-action on $Y_{k_h}$. Notice that $P_s$ is diffeomorphic to the whole contact sphere $Y$, and in particular is path-connected.

We enumerate the critical values of $\Psi$ (that is, the elements of the action spectrum $\sigma(Y)$) in increasing order as
$\sigma_1<\sigma_2<\sigma_3<\ldots$,
and denote by 
\[K_j:=\crit(\Psi)\cap\Psi^{-1}(\sigma_j)\] 
the corresponding critical set. Notice that $K_j=P_j$ for all $j=1,\ldots,s$, and more generally each $K_j$ is of the form
\begin{align*}
 K_j 
 = 
 P_{h_j}^{m_j} 
 = 
 \big\{
 \dot\gamma^{m_j}\ \big|\ \dot\gamma\in P_{h_j}
 \big\}
\end{align*}
for some $h_j\in\{1,\ldots,s\}$ and $m_j>0$. Since the $S^1$ action on the $K_j$'s is locally free, their rational $S^1$-equivariant cohomology is nothing but the rational cohomology of their quotient by $S^1$. This implies
\begin{align}
\label{e:same_rational_cohomology}
 H^*_{S^1}(K_j;\Q) 
 \cong 
 H^*(K_j/S^1;\Q)
 \cong 
 H^*(P_{h_j}/S^1;\Q)
 \cong
  H^*_{S^1}(P_{h_j};\Q).
\end{align}
We denote by $\pi_0(K_j)$ the family of path-connected components $K\subset K_j$, and we set
\begin{align*}
 \iota_0(j) & :=\min\big\{ \ind(K)\ \big|\ K\in\pi_0(K_j) \big\},\\
 \iota_1(j) & :=\max\big\{ \ind(K)+\nul(K)-1\ \big|\ K\in\pi_0(K_j) \big\}.
\end{align*}
We recall that, by Lemma~\ref{l:GGM}, the indices $\iota_0(j)$ and $\iota_1(j)$ are all even. We also recall that $\nul(K)=\dim(K)$ for each $K\in\pi_0(K_j)$, since the Clarke action functional $\Psi$ is Morse-Bott. 

We choose a sequence of positive real numbers $b_j$, for $j\geq0$, such that
\begin{align*}
0<b_0 < \sigma_1 < b_1 < \sigma_2 < b_2 < \sigma_3 < b_3 < \ldots,
\end{align*}
and we set 
\[\Lambda_j:=\Lambda^{<b_j}=\Psi^{-1}(0,b_j).\]
Lemma~\ref{l:local_cohomology}(i) implies that the $S^1$-equivariant local cohomology of the critical sets $K_j$ is given by
\begin{align*}
H^*_{S^1}(\Lambda_j,\Lambda_{j-1};\Z)\cong \bigoplus_{K\in\pi_0(K_j)} \!\!\! H^{*-\ind(K)}_{S^1}(K;\Z) .
\end{align*}

Theorem~\ref{t:perfect} will be a consequence of the next proposition. In the case of Besse Riemannian geodesic flows, an analogous statement was established in \cite[Prop.~5.4]{Radeschi:2017dz}. In the following, we shall denote by $H^{\odd}_{S^1}$ the cohomology in odd degrees, i.e.
\begin{align*}
 H^{\odd}_{S^1}(\cdot):=\bigoplus_{d\geq0} H^{2d+1}_{S^1}(\cdot).
\end{align*}

\begin{prop}\label{p:zero_odd_cohomology}
For each $j\geq 1$ we have $H^{\odd}_{S^1}(K_j;\Q)=0$.
\end{prop}

\begin{proof}
By~\eqref{e:same_rational_cohomology}, it is enough to show that 
$H^{\odd}_{S^1}(P_{h};\Q)=0$ for all $h=1,\ldots,s$.
We will prove the proposition by contradiction: let us assume that  
\begin{align}
\label{e:odd_cohomology} 
H^{\odd}_{S^1}(P_{k};\Q)\neq0
\end{align}
for some $k\in\{1,\ldots,s\}$.
We fix the minimum such $k$, so that if $k>1$ we have 
\begin{align*}
H^{\odd}_{S^1}(P_{1};\Q)\cong\ldots\cong H^{\odd}_{S^1}(P_{k-1};\Q)=0.
\end{align*}

At every degree $d$, the $S^1$-equivariant cohomology groups $H^d_{S^1}(P_{h};\Z)$ are finitely generated. If $d$ is larger than the maximal dimension of $P_h$, the ordinary cohomology groups $H^d(P_{h};\Z)$ vanish, and the cup product with the Euler class $e$ in the Gysin sequence
\begin{equation*}
\begin{tikzcd}[row sep=small]
 H^{d}(P_{h};\Z) \arrow[r] \equaldown
 &
 H^{d-1}_{S^1}(P_{h};\Z) \arrow[r, "\smallsmile e", "\cong"']
 &
 H^{d+1}_{S^1}(P_{h};\Z) \arrow[r]
 &
 H^{d+1}(P_{h};\Z) \equaldown\\
 0 & & & 0
\end{tikzcd}
\end{equation*}
is an isomorphism. Therefore, $H^*_{S^1}(P_{h};\Z)$ is finitely generated as a ring, and in particular any prime number larger than some $p_0$ does not divide the order of any torsion element of $H^{*}_{S^1}(P_{h};\Z)$ for all $h=1,\ldots,s$. 

Let us consider the critical manifold $P_k$ satisfying~\eqref{e:odd_cohomology}. We fix a compact neighborhood $[c-\delta,c+\delta]\subset(0,\infty)$ of its critical value $c:=\sigma_k=\Psi(P_k)$ that does not contain other critical values of $\Psi$. Notice that $P_k^p\subseteq K_{r}$ for the integer $r$ such that \[\sigma_r=pc,\] and the inclusion is an actual equality if $p$ is a large enough prime. We fix the prime number $p>p_0$ to be large enough so that $P_k^p=K_r$ and 
\begin{align*}
 p(c-\delta)\leq \sigma_{r_1} < \sigma_r < \sigma_{r_2}\leq p(c+\delta),
\end{align*}
for some critical values $\sigma_{r_1}$ and $\sigma_{r_2}$ that are multiples of the common period $\tau$.

The proof of the proposition will be based on a rather subtle analysis of the $p$-torsion of the cohomology groups $H^*_{S^1}(\Lambda_j;\Z)$, which will produce some non-trivial cohomology $H_{S^1}^{d}(\Lambda;\Z)$ in degree $d=\iota_1(r_2)+1$; since $d$ is odd, this will contradict Lemma~\ref{l:domain_Clarke}. It will therefore be convenient to discard all the torsion of order that is not divisible by $p$. We will do it by considering the coefficient ring
\begin{align*}
R:=\big\{\tfrac ab  \big|\ a,b\in\Z,\ \tfrac bp\not\in\Z\big\}\subset\Q.
\end{align*}
Notice that, since $p>p_0$, each $H^{*}_{S^1}(P_{h};R)$  is torsion-free, and in particular vanishes in degrees larger than or equal to the maximal dimension of $P_h$. From now on, all the cohomology groups will have coefficients in $R$ unless we specify otherwise, and we will remove $R$ from the notation.

Since $\sigma_{r_1}$ and $\sigma_{r_2}$ are multiples of the common period $\tau$, the associated critical manifolds $K_{r_1}$ and $K_{r_2}$ are both diffeomorphic to $Y$, and $r_1<r<r_2$. In particular, $K_{r_1}$ and $K_{r_2}$ are path-connected, and we have
\[
\iota_1(r_1)-\iota_0(r_1)
=
\iota_1(r_2)-\iota_0(r_2)
=
\dim(K_{r_1})-1
=
\dim(K_{r_2})-1
=
2n-2.
\] 
The Morse index formulas of Lemma~\ref{l:Morse} imply: 
\begin{itemize}
\item[\textbf{(a)}] If $j<r_1$ then $\iota_1(j) + 2 \leq \iota_0(r_1)$,\vspace{2pt}
\item[\textbf{(b)}] If $r_1<j<r_2$ then $\iota_1(r_1) + 2 = \iota_0(r_1) + 2n \leq \iota_0(j)$ and $\iota_1(j)+2\leq\iota_0(r_2)$,\vspace{2pt}
\item[\textbf{(c)}] If $j>r_2$ then $\iota_1(r_2) + 2 = \iota_0(r_2) + 2n \leq \iota_0(j)$.
\end{itemize}
By Proposition~\ref{p:torsion}(i), every critical set $K_j=P_{h_j}^{m_j}$ with $j\in\{1,\ldots,r_2\}\setminus\{r\}$ satisfies one of the two following conditions:
\begin{itemize}
\item ${m_j}/p\in\N$, and therefore $h_j<k$, $H^{\odd}_{S^1}(K_j;\Q)\cong H^{\odd}_{S^1}(P_{h_j};\Q)=0$, and $H^{\odd}_{S^1}(K_j)$ is torsion-free;\vspace{2pt}

\item ${m_j}/p\not\in\N$, and therefore $H^*_{S^1}(K_j)$ is torsion-free.

\end{itemize}
In both cases, $H^{\odd}_{S^1}(K_j)$ is torsion-free. Since 
\begin{align*}
\rank_{R} H^{2d+1}_{S^1}(K_j)
=
\rank_{\Q} H^{2d+1}_{S^1}(K_j;\Q) 
=
\rank_{\Q} H^{2d+1}(K_j/S^1;\Q),
\end{align*}
and since all Morse indices are even, we infer
\begin{align*}
 H^{2d+1-\ind(K)}_{S^1}(K) = 0,
 \qquad\forall j\in\{1,\ldots,r_2\}\setminus\{r\},\ \ K\in\pi_0(K_j),\ \ 2d+1>\iota_1(j).
\end{align*}
This, together with Lemma~\ref{l:local_cohomology}(ii) and the inequality \textbf{(a)}, implies
\begin{align*}
 H^{2d+1}_{S^1}(\Lambda_{r_1-1}) = 0,\qquad \forall\, 2d+1\geq\iota_0(r_1)-1.
\end{align*}
As for the ordinary cohomology, the inequality \textbf{(a)} implies 
\begin{align*}
 H^{d-\ind(K)}(K)=0 \qquad\forall j\in\{1,\ldots,r_1-1\},\ \ K\in\pi_0(K_j),\ \ d\geq\iota_0(r_1),
\end{align*}
which, together with Lemma~\ref{l:local_cohomology}(ii), provides $H^{d}(\Lambda_{r_1-1})=0$ for all $d\geq \iota_0(r_1)$.
Therefore, in the Gysin sequence
\begin{align*}
 \ldots 
 \ttoup
 H^{*}_{S^1}(\Lambda_{r_1-1}) 
 \ttoup^{\smallsmile e}  
 H^{*+2}_{S^1}(\Lambda_{r_1-1}) 
 \ttoup
 H^{*+2}(\Lambda_{r_1-1}) 
 \ttoup
 \ldots 
\end{align*}
the cup product with the Euler class $e$ is a surjective homomorphism
\begin{align*}
H^{2d}_{S^1}(\Lambda_{r_1-1})\eepi^{\smallsmile e} H^{2d+2}_{S^1}(\Lambda_{r_1-1}),
\qquad\forall\, 2d\geq\iota_0(r_1)-2.
\end{align*}
Here and in the following, we denote by a two-head arrow $\epi$ a surjective homomorphism.

Since $K_{r_1}\cong S^{2n-1}$, the Gysin sequence 
\begin{align*}
 H^{*+1}(S^{2n-1};\Q) 
 \toup
 H^{*}_{S^1}(K_{r_1};\Q) 
 \ttoup^{\smallsmile e}  
 H^{*+2}_{S^1}(K_{r_1};\Q) 
 \toup
 H^{*+2}(S^{2n-1};\Q) 
\end{align*}
readily implies that $H^{*}_{S^1}(K_{r_1};\Q)\cong\Q[e]/(e^n)$, where the Euler class $e$ is the generator of $H^{2}_{S^1}(K_{r_1};\Q)$. Since $m_{r_1}/p\not\in\N$, Proposition~\ref{p:torsion}(i) allows us to draw the same conclusion with the coefficients in $R$, i.e.
\begin{align*}
 H^{*}_{S^1}(K_{r_1})\cong \frac {R[e]}{(e^n)},
\end{align*}
where $e$ is now the generator of $H^{2}_{S^1}(K_{r_1})$. Clearly, all the assertions of this paragraph hold for $K_{r_2}$ as well.

By Lemma~\ref{l:local_cohomology}(i) we have $H^{*}_{S^1}(\Lambda_{r_1},\Lambda_{r_1-1})\cong H_{S^1}^{*-\iota_0(r_1)}(K_{r_1})$, and in particular \[H^{\odd}_{S^1}(\Lambda_{r_1},\Lambda_{r_1-1})=0.\] For each $2d\geq\iota_0(r_1)-2$ the long exact sequence of the inclusion $\Lambda_{r_1-1}\subset\Lambda_{r_1}$ gives a commutative diagram
\begin{equation*}
\begin{tikzcd}[row sep=large]
 H_{S^1}^{2d-\iota_0(r_1)}(K_{r_1}) \arrow[r] \arrow[d, twoheadrightarrow, "\smallsmile e"]
 &
 H^{2d}_{S^1}(\Lambda_{r_1}) \arrow[r] \arrow[d, "\smallsmile e"]
 &
 H^{2d}_{S^1}(\Lambda_{r_1-1}) \arrow[r] \arrow[d, twoheadrightarrow, "\smallsmile e"]
 &
 0 \\
 H_{S^1}^{2d+2-\iota_0(r_1)}(K_{r_1}) \arrow[r] 
 &
 H^{2d+2}_{S^1}(\Lambda_{r_1}) \arrow[r]
 &
 H^{2d+2}_{S^1}(\Lambda_{r_1-1}) \arrow[r]
 &
 0
\end{tikzcd}
\end{equation*}
whose rows are exact, and whose first and third vertical homomorphisms are surjective. Simple diagram chasing allows us to conclude that the second vertical homomorphism is surjective as well, i.e.
\begin{align*}
H^{2d}_{S^1}(\Lambda_{r_1})\eepi^{\smallsmile e} H^{2d+2}_{S^1}(\Lambda_{r_1}),
\qquad\forall\, 2d\geq\iota_0(r_1)-2.
\end{align*}

For every $j>r_1$ and for every path-connected component $K\in\pi_0(K_j)$, the inequality~\textbf{(b)} implies $\iota_0(r_1)-1-\ind(K)<0$. By Lemma~\ref{l:local_cohomology}(ii), we infer
\begin{align*}
H^{\iota_0(r_1)-1}_{S^1}(\Lambda,\Lambda_{r_1})\cong \bigoplus_{\scriptsize
  \begin{array}{@{}c@{}}
    j>r_1 \vspace{1pt}\\ 
    K\in\pi_0(K_j) 
  \end{array}
} \!\!\!\!\!  H^{\iota_0(r_1)-1-\ind(K)}_{S^1}(K) =0,
\end{align*}
and thus we have a surjective homomorphism
\begin{align*}
 H^{\iota_0(r_1)-2}_{S^1}(\Lambda)
 \epi
 H^{\iota_0(r_1)-2}_{S^1}(\Lambda_{r_1}).
\end{align*}
By fitting this homomorphism into the commutative diagram
\[\begin{tikzcd}[row sep=large]
 H^{\iota_0(r_1)-2}_{S^1}(\Lambda) 
 \arrow[rr, "\smallsmile e^d", "\cong"'] 
 \arrow[d, twoheadrightarrow]
 &&
 H^{\iota_0(r_1)+2d-2}_{S^1}(\Lambda) 
 \arrow[d]
 \\
 H^{\iota_0(r_1)-2}_{S^1}(\Lambda_{r_1}) 
 \arrow[rr, twoheadrightarrow, "\smallsmile e^d"] 
 &&
 H^{\iota_0(r_1)+2d-2}_{S^1}(\Lambda_{r_1})
\end{tikzcd}\]
we infer that the right vertical homomorphism is surjective, i.e.
\begin{align}\label{e:step1}
 H^{2d}_{S^1}(\Lambda)\epi H^{2d}_{S^1}(\Lambda_{r_1}),
 \qquad\forall\, 2d\geq \iota_0(r_1).
\end{align}

Since $c$ is the only critical value of $\Psi$ in $[c-\delta,c+\delta]$, if $r_1<j<r_2$ and $j\neq r$ we must have ${m_j}/p\not\in\N$. Therefore $H^{*}_{S^1}(K_j)$ is torsion-free, and for each path-connected component $K\in\pi_0(K_j)$ the cohomology $H^{*}_{S^1}(K)$ must vanish in degrees larger than or equal to $\dim(K)$. In particular, the inequalities \textbf{(b)} and \textbf{(c)} imply
\begin{align*}
 H^{d-\ind(K)}_{S^1}(K)=0,\quad\forall j\in\{r_1+1,\ldots,r_2-1\}\setminus\{r\},\  K\in\pi_0(K_j),\ d\geq\iota_0(r_2)-1,
\end{align*}
and Lemma~\ref{l:local_cohomology}(ii) provides
\begin{align}
\label{e:zero_rel_cohom}
 H^{d}_{S^1}(\Lambda_{r-1},\Lambda_{r_1})\cong H^{d}_{S^1}(\Lambda_{r_2-1},\Lambda_{r})=0,\qquad\forall d\geq\iota_0(r_2)-1.
\end{align}
The vanishing of the first of these two cohomology groups implies that we have an isomorphism
\begin{align*}
 H^d(\Lambda_{r-1})\ttoup^{\cong} H^d(\Lambda_{r_1}),\qquad\forall d\geq\iota_0(r_2)-1,
\end{align*}
and together with~\eqref{e:step1}, this implies that we have a surjective homomorphism
\begin{align*}
 H^{2d}_{S^1}(\Lambda)\epi H^{2d}_{S^1}(\Lambda_{r-1}),
 \qquad\forall\, 2d\geq \iota_0(r_2).
\end{align*}
This homomorphism factorizes through  $H^{2d}_{S^1}(\Lambda_r)$, i.e.
\begin{equation*}
\begin{tikzcd}[row sep=large]
 H^{2d}_{S^1}(\Lambda) 
 \arrow[r, twoheadrightarrow] 
 \arrow[d]
 &
 H^{2d}_{S^1}(\Lambda_{r-1}) 
 \\
 H^{2d}_{S^1}(\Lambda_{r}) 
 \arrow[ru] 
\end{tikzcd}
\end{equation*}
Therefore, we have a surjective homomorphism
\begin{align*}
 H^{2d}_{S^1}(\Lambda_r)\epi H^{2d}_{S^1}(\Lambda_{r-1}),
 \qquad\forall\, 2d\geq \iota_0(r_2).
\end{align*}

We now consider the ``anomalous'' critical manifold $K_r$, which is of the form $K_r=P_k^p$. 
Since $H^{\odd}(P_k;\Q)\neq0$, Proposition~\ref{p:torsion}(ii) implies that, for some path-connected component $K'\in\pi_0(K_r)$,
\[H^{2d+1}_{S^1}(K')_{\tor}\neq0,\qquad \forall 2d+1\geq\dim(K').\] 
Here, the subscript ``tor'' denotes the torsion subgroup.
We fix the odd degree
\[d:=\iota_1(r_2)+1.\] 
The inequality \textbf{(b)} implies $d-\ind(K')\geq\dim(K')$, and thus \[H^{d-\ind(K')}_{S^1}(K')_{\tor}\neq0.\] Since
\[
H^{d}_{S^1}(\Lambda_r,\Lambda_{r-1}) 
\cong 
\bigoplus_{K\in\pi_0(K_r)} \!\!\!
H^{d-\ind(K)}_{S^1}(K)
\]
according to Lemma~\ref{l:local_cohomology}(i), the long exact sequence
\begin{equation*}
 \begin{tikzcd}[row sep=large]
 H^{d-1}_{S^1}(\Lambda_r) 
 \arrow[r, twoheadrightarrow] 
 &
 H^{d-1}_{S^1}(\Lambda_{r-1}) 
 \arrow[r]
 &
 \displaystyle\bigoplus_{K\in\pi_0(K_r)} \!\!\! H^{d-\ind(K)}_{S^1}(K) 
 \arrow[r, "a^*"]
 &
 H^{d}_{S^1}(\Lambda_{r}) 
\end{tikzcd}
\end{equation*}
readily implies that $a^*$ is injective, and in particular 
\[H^{d}_{S^1}(\Lambda_{r})_{\tor}\neq0.\] 
Notice that $H^{d}_{S^1}(\Lambda_{r_2-1})\cong H^{d}_{S^1}(\Lambda_{r})$ according to~\eqref{e:zero_rel_cohom}. Moreover, 
\begin{align*}
 H^d_{S^1}(\Lambda_{r_2},\Lambda_{r_2-1}) & \cong H^{d-\iota_0(r_2)}_{S^1}(K_{r_2}) = 0,\\ 
 H^{d+1}_{S^1}(\Lambda_{r_2},\Lambda_{r_2-1}) & \cong H^{d+1-\iota_0(r_2)}_{S^1}(K_{r_2}) = 0.
\end{align*}
Therefore $H^{d}_{S^1}(\Lambda_{r_2})\cong H^{d}_{S^1}(\Lambda_{r_2-1})$. Overall, we showed that 
\begin{align*}
H^{d}_{S^1}(\Lambda_{r_2})_{\tor}\neq0. 
\end{align*}

Let us finally consider the higher critical sets $K_j$, for $j>r_2$. Inequality \textbf{(c)} implies $\iota_0(j)>d$, and therefore 
\begin{align*}
H^{d}_{S^1}(\Lambda_j,\Lambda_{j-1}) &\cong \bigoplus_{K\in\pi_0(K_j)} \!\!\! H^{d-\ind(K)}_{S^1}(K)=0,\\
H^{d+1}_{S^1}(\Lambda_j,\Lambda_{j-1}) &\cong \bigoplus_{K\in\pi_0(K_j)} \!\!\!  H^{d+1-\ind(K)}_{S^1}(K)\cong
\underbrace{R\oplus\ldots\oplus R}_{\times q_j},
\end{align*}
where $q_j$ is the number of path-connected components $K\in\pi_0(K_j)$ such that $\ind(K)=d+1$. Notice in particular that $H^{d+1}_{S^1}(\Lambda_j,\Lambda_{j-1})$ is torsion-free. We claim that $H^{d}_{S^1}(\Lambda_{j})_{\tor}\neq0$ for all $j\geq r_2$. Indeed, we already know that this holds for $j=r_2$. Assume that it holds for degree $j$, and consider the long exact sequence
\begin{equation*}
 \begin{tikzcd}[row sep=small]
 H^{d}_{S^1}(\Lambda_{j+1},\Lambda_j) 
 \arrow[r] \equaldown
 &
 H^{d}_{S^1}(\Lambda_{j+1}) 
 \arrow[r]
 &
 H^{d}_{S^1}(\Lambda_{j}) 
 \arrow[r, "\delta"]
 &
 H^{d+1}_{S^1}(\Lambda_{j+1},\Lambda_j) \\
 0 &&&
\end{tikzcd}
\end{equation*}
Since $H^{d+1}_{S^1}(\Lambda_{j+1},\Lambda_j)$ is torsion-free, the torsion subgroup $H^{d}_{S^1}(\Lambda_{j})_{\tor}$ is in the kernel of the connecting homomorphism $\delta$, and therefore the inclusion induces an isomorphism
\begin{align}
\label{e:nontrivial_torsion}
  H^{d}_{S^1}(\Lambda_{j+1})_{\tor} \toup^{\cong}
  H^{d}_{S^1}(\Lambda_{j})_{\tor}\neq0.
\end{align}
Consider now the inverse limit of the groups $H^{d}_{S^1}(\Lambda_{j})$ as $j\to\infty$, and recall that the cohomology admits a surjective homomorphism
\begin{align}
\label{e:inverse_limit}
 H^*_{S^1}(\Lambda) \epi \varprojlim H^*_{S^1}(\Lambda_j),
\end{align}
see, e.g., \cite[Theorem 3F.8]{Hatcher:2002dt}. Equations~\eqref{e:nontrivial_torsion} and~\eqref{e:inverse_limit} imply that $H^d_{S^1}(\Lambda)$ is non-trivial in the odd degree $d$, contradicting Lemma~\ref{l:domain_Clarke}.
\end{proof}

\subsection{Proof of the main theorems}
By Lemma~\ref{l:local_cohomology}(i), we have isomorphisms of $H^*(BS^1;\Q)$-modules
\begin{align}
\label{e:iso_modules_final_proof}
H^*_{S^1}(\Lambda_j,\Lambda_{j-1};\Q)\cong \bigoplus_{K\in\pi_0(K_j)} \!\!\! H^{*-\ind(K)}_{S^1}(K;\Q).
\end{align}
Since every Morse index $\ind(K)$ is even, Equation~\eqref{e:iso_modules_final_proof} and Proposition~\ref{p:zero_odd_cohomology} imply that 
\begin{align}
\label{e:zero_odd_cohomology_final_proof}
H^{\odd}_{S^1}(\Lambda_j,\Lambda_{j-1};\Q)\cong H^{\odd}_{S^1}(K_j;\Q)=0.
\end{align}
This allows us to apply Morse's lacunary principle, which proves Theorem~\ref{t:perfect}: the long exact sequences of the inclusions $\Lambda_{j-1}\subset\Lambda_j$ split in short exact sequences
\begin{align*}
 0
 \toup 
 H^*_{S^1}(\Lambda_j,\Lambda_{j-1};\Q)
 \toup 
 H^*_{S^1}(\Lambda_j;\Q)
 \toup 
 H^*_{S^1}(\Lambda_{j-1};\Q)
 \toup 
 0.
\end{align*}
Therefore
\begin{align*}
H^*_{S^1}(\Lambda_j;\Q) \cong \bigoplus_{1\leq k\leq j} H^*_{S^1}(\Lambda_k,\Lambda_{k-1};\Q).
\end{align*}
By Equation~\eqref{e:iso_modules_final_proof}, the local cohomology $H^d_{S^1}(\Lambda_j,\Lambda_{j-1};\Q)$ vanishes for $d<\iota_0(j)$. Since $\iota_0(j)\to\infty$ as $j\to\infty$, for any degree $d$ and for any large enough $j$ the inclusion induces an isomorphism
\begin{align*}
  H^{d}_{S^1}(\Lambda_{j+1};\Q) 
 \toup^{\cong}
 H^{d}_{S^1}(\Lambda_{j};\Q). 
\end{align*}
Therefore, 
\begin{align}
\label{e:perfect}
H^*_{S^1}(\Lambda;\Q) \cong \varprojlim H^*_{S^1}(\Lambda_j;\Q) \cong \varprojlim\bigoplus_{j\geq1} H^*_{S^1}(\Lambda_j,\Lambda_{j-1};\Q). 
\end{align}
By Lemma~\ref{l:domain_Clarke}, we have $H^*_{S^1}(\Lambda;\Q)=\Q[e]$, where $e\in H^2_{S^1}(\Lambda;\Q)$ is the Euler class. This, together with~\eqref{e:perfect}, implies that each local cohomology group in a given even degree $H^{2d}_{S^1}(\Lambda_j,\Lambda_{j-1};\Q)$ has rank at most one, and conversely for every even degree $2d$ there exists a unique $j$ such that $H^{2d}_{S^1}(\Lambda_j,\Lambda_{j-1};\Q)\neq0$, and it must be $H^{2d}_{S^1}(\Lambda_j,\Lambda_{j-1};\Q)\cong\Q$.

Consider the path-connected components $K'_j,K''_j\in\pi_0(K_j)$ such that 
\begin{align*}
\iota_0(j) & = \ind(K_j'),
\\
\iota_1(j) & = \ind(K_j'')+\nul(K_j'')-1=\ind(K_j'')+\dim(K_j'')-1.
\end{align*}
We have not proved yet that the critical sets $K_j$ are path-connected; a priori, $K_j'$ and $K_j''$ may be distinct path-connected components. Equation~\eqref{e:iso_modules_final_proof} and the conclusion of the previous paragraph imply that
\begin{gather*}
H^{\iota_0(j)}_{S^1}(\Lambda_j,\Lambda_{j-1};\Q)\cong H^{0}_{S^1}(K_{j}';\Q)\cong\Q,\\
H^{\iota_1(j)}_{S^1}(\Lambda_j,\Lambda_{j-1};\Q)\cong H^{\dim(K_j'')-1}_{S^1}(K_j'';\Q)\cong H^{\dim(K_j'')-1}(K_j''/S^1;\Q)\cong\Q.
\end{gather*}
Therefore
\begin{align}
\label{e:not_in_ker}
 e^{\iota_1(j)/2}\not\in \ker\big(H^*_{S^1}(\Lambda;\Q)\to H^*_{S^1}(\Lambda_j;\Q)\big),
\end{align}
but
\begin{align}
\label{e:in_ker}
 e^{\iota_0(j)/2}\in \ker\big(H^*_{S^1}(\Lambda;\Q)\to H^*_{S^1}(\Lambda_{j-1};\Q)\big).
\end{align}
We set
\begin{align*}
d_j:=\frac{\iota_1(j)-\iota_0(j)}{2}+1.
\end{align*}

\begin{lem}
\label{l:rational_homology_sphere}
Every critical set $K_j$ is path-connected, $d_j=(\dim(K_j)+1)/2$, and $H^*_{S^1}(K_j;\Q)=\Q[e]/(e^{d_j})$, where $e\in H^2_{S^1}(K_j;\Q)$ is the Euler class.
\end{lem}

\begin{proof}
We consider the commutative diagram
\begin{equation*}
\begin{tikzcd}[row sep=large]
 H^{\iota_0(j)}_{S^1}(\Lambda_j,\Lambda_{j-1};\Q) 
 \arrow[r,"a^*",hookrightarrow] 
 \arrow[d,"\smallsmile e^{d_j-1}"]
 &
 H^{\iota_0(j)}_{S^1}(\Lambda_j;\Q) 
 \arrow[r,"b^*"] 
 \arrow[d,"\smallsmile e^{d_j-1}"]
 &
 H^{\iota_0(j)}_{S^1}(\Lambda_{j-1};\Q) 
 \\
 H^{\iota_1(j)}_{S^1}(\Lambda_j,\Lambda_{j-1};\Q) 
 \arrow[r,"c^*",hookrightarrow] 
 &
 H^{\iota_1(j)}_{S^1}(\Lambda_{j};\Q)
 \arrow[r,"d^*"]
 &
 H^{\iota_1(j)}_{S^1}(\Lambda_{j-1};\Q)
\end{tikzcd}
\end{equation*}
whose rows are exact. Equations~\eqref{e:not_in_ker} and \eqref{e:in_ker} imply that $e^{\iota_0(j)/2}\in\ker(b^*)\setminus\{0\}$, and therefore \[e^{\iota_0(j)/2}=a^*(\tilde e)\] 
for some $\tilde e\in H^{\iota_0(j)}_{S^1}(\Lambda_j,\Lambda_{j-1};\Q)$. The diagram implies that 
\[c^*(e^{d_j-1}\smallsmile\tilde e)=e^{\iota_1(j)/2}.\] In particular 
\begin{align*}
e^{d}\smallsmile\tilde e\neq0\mbox{ in }H^{\iota_0(j)+2d}_{S^1}(\Lambda_j,\Lambda_{j-1};\Q),
\quad\forall d=0,\ldots,d_j-1. 
\end{align*}
Under the $H^*(BS^1;\Q)$-modules isomorphism~\eqref{e:iso_modules_final_proof}, the cohomology class $e^{d}\smallsmile\tilde e$ is mapped to $e^d\in H^{2d}_{S^1}(K_j';\Q)$ for each $d=0,\ldots,(\iota_1(j)-\iota_0(j))/2$. This, together with~\eqref{e:zero_odd_cohomology_final_proof}, implies that 
\begin{align*}
H^*_{S^1}(K_j';\Q)=\frac{\Q[e]}{(e^{d_j})}.
\end{align*}
In particular
$\dim(K_j') = 2d_j-1=\iota_1(j)-\iota_0(j)+1$.
The critical set $K_j$ must then be path-connected, and thus coincide with $K_j'$. Indeed, if there were another path-connected component $K\in\pi_0(K_j)$ with $K\neq K_j'$, its even Morse index would satisfy \[\iota_0(j)\leq\ind(K)\leq\iota_1(j);\] 
however, both $H^{0}_{S^1}(K;\Q)\cong\Q$ and $H^{\ind(K)-\iota_0(j)}_{S^1}(K_j';\Q)\cong\Q$ would contribute to the local cohomology \[H^{\ind(K)}_{S^1}(\Lambda_j,\Lambda_{j-1};\Q)\] via the isomorphism~\eqref{e:iso_modules_final_proof}, contradicting the fact that $H^{\ind(K)}_{S^1}(\Lambda_j,\Lambda_{j-1};\Q)$ has rank~1.
\end{proof}

We recall once again that the non-empty strata $Y_k=\fix(\psi^{\tau/k})$ are closed contact submanifolds of $Y$, and in particular are orientable. Every such $Y_k$ is equivariantly diffeomorphic to its corresponding critical set $K_j$, with $\sigma_j=\tau/k$, via the map~\eqref{e:equiv_diff_crit_strat}. Therefore, in Theorem~\ref{t:main}, we can equivalently replace $Y_k$ with the critical set $K_j$.

\begin{proof}[Proof of Theorem~\ref{t:main}]
Since $K_j$ is a path-connected orientable closed manifold, we have
\begin{align*}
H^{0}(K_j;\Z)\cong H^{\dim(K_j)}(K_j;\Z)\cong\Z.
\end{align*}
Analogously, $H^{\dim(K_j)}(K_j;\Q)\cong\Q$.
Lemma~\ref{l:rational_homology_sphere}, together with the Gysin sequence 
\begin{align*}
 \ldots 
 \toup
 H^{*+1}(K_j;\Q) 
 \ttoup
 H^{*}_{S^1}(K_j;\Q) 
 \ttoup^{\smallsmile e}  
 H^{*+2}_{S^1}(K_j;\Q) 
 \ttoup
 H^{*+2}(K_j;\Q) 
 \toup
 \ldots 
\end{align*}
implies that $K_j$ is a rational homology sphere of dimension $\dim(K_j)=2d_j-1$, i.e.
\begin{align*}
 H^*(K_j;\Q)\cong H^*(S^{\dim(K_j)};\Q).
\end{align*}
Since $K_j$ is path-connected,  we have \[\iota_0(j)=\ind(K_j),\qquad \iota_1(j)=\iota_0(j)+\dim(K_j)-1.\] 
Equations~\eqref{e:not_in_ker} and~\eqref{e:in_ker} imply implies that $\iota_0(j+1)>\iota_1(j)$, and therefore 
\[\iota_0(j+1)=\iota_1(j)+2.\]
By Lemma~\ref{l:local_cohomology}(ii), the local cohomology $H^d(\Lambda_j,\Lambda_{j-1};\Z)$ can be non-trivial only if $\iota_0(j)\leq d\leq\iota_1(j)+1$, and
\begin{align*}
 H^d(\Lambda,\Lambda_j;\Z)\cong H^d(\Lambda_{j-1};\Z)\cong0,\qquad \forall d\in\{\iota_0(j)+1,\ldots,\iota_1(j)\}.
\end{align*}
This, together with Lemmas~\ref{l:local_cohomology}(i) and~\ref{l:domain_Clarke}, implies
\begin{align*}
 H^d(K_j;\Z) \cong H^{d+\iota_0(j)}(\Lambda_j,\Lambda_{j-1};\Z)
 \cong 
 H^{d+\iota_0(j)}(\Lambda;\Z)\cong0,\\
 \forall d\in\{1,\ldots,\dim(K_j)-1\}.
\end{align*}
Summing up, we proved that $H^*(K_j;\Z)\cong H^*(S^{\dim(K_j)};\Z)$. By the universal coefficient theorem, this is equivalent to $H_*(K_j;\Z)\cong H_*(S^{\dim(K_j)};\Z)$
\end{proof}

\begin{proof}[Proof of Theorem~\ref{t:spec}]
For each $j\geq1$, Equation~\eqref{e:not_in_ker} implies that $c_{\iota_1(j)/2}(Y)\leq \sigma_j$, whereas Equation~\eqref{e:in_ker} implies that $c_{\iota_0(j)/2}(Y)> \sigma_{j-1}$. Overall,
\[\sigma_j=c_{\iota_0(j)/2}(Y)=c_{\iota_1(j)/2}(Y)<c_{1+\iota_1(j)/2}(Y),\]
and, by Lemma~\ref{l:rational_homology_sphere}, $d_j=1+(\iota_1(j)-\iota_0(j))/2=(\dim(K_j)+1)/2$. Therefore, the sequence of Ekeland-Hofer spectral invariants $c_1(Y), c_2(Y),c_3(Y),\ldots$
is precisely
\[
\underbrace{\sigma_1,\ldots,\sigma_1}_{\times d_1},\underbrace{\sigma_2,\ldots,\sigma_2}_{\times d_2},\underbrace{\sigma_3,\ldots,\sigma_3}_{\times d_3},\ldots
\qedhere
\]
\end{proof}

\appendix

\section{Euler class of fibered products}
\label{a:Euler}

Let $S^1\subset\C$ be the unit circle in the complex plane.
If $X$ is a space equipped with an $S^1$ action $(e^{i\theta},x)\to e^{i\theta}\cdot x$, for any integer $q>0$ we denote by $X^q$ the same space equipped with the $S^1$ action $(e^{i\theta},x)\to e^{iq\theta}\cdot x$. In the proof of Proposition~\ref{p:torsion}, we will need the following elementary fact.

\begin{lem}
\label{l:Euler}
Consider two principal $S^1$ bundles $C_1\to B_1$ and $C_2\to B_2$, an integer $q>0$, and the fibered product $C_1^q\times_{S^1}C_2$ equipped with the $S^1$-action
\begin{align*}
e^{i\theta}\cdot[c_1,c_2]=[e^{iq\theta}c_1,c_2]=[c_1,e^{-i\theta}c_2],\qquad\forall e^{i\theta}\in S^1,\ [c_1,c_2]\in C_1^q\times_{S^1}C_2,
\end{align*}
so that the quotient projection $C_1^q\times_{S^1}C_2\to B_1\times B_2$ is a principal $S^1$-bundle. Its Euler class is
\begin{align*}
 e_1\otimes1 - q\cdot1\otimes e_2\in H^*(B_1\times B_2;\Z),
\end{align*}
where $e_i\in H^2(B_i;\Z)$ is the Euler class of $C_i\to B_i$.
\end{lem}

\begin{proof}
The statement is a bundle version of the following linear algebra remark. Consider the fibered product $(S^1)^q\times_{S^1} S^1$ equipped with the $S^1$-action
\begin{align*}
e^{i\theta}\cdot[e^{is},e^{it}]=[e^{i(q\theta+s)},e^{it}]=[e^{is},e^{i(-\theta+t)}],
\end{align*}
for all $e^{i\theta}\in S^1$ and $[e^{is},e^{it}]\in(S^1)^q\times_{S^1} S^1$. Next, consider the tensor product $\C\otimes\overline{\C}{}^{\otimes q}$, where the complex conjugacy means that
\begin{align*}
z\otimes 1\otimes\ldots\otimes 1 
= 
1\otimes \overline z \otimes\ldots\otimes 1
= \ldots =
1\otimes1\otimes \ldots\otimes\overline z,\qquad\forall z\in\C.
\end{align*}
The unit circle $S\subset\C\otimes\overline{\C}{}^{\otimes q}$ is diffeomorphic to $(S^1)^q\times_{S^1} S^1$ via the $S^1$-equivariant map
\begin{align*}
 \psi:(S^1)^q\times_{S^1} S^1\toup^{\cong} S,\qquad \psi([e^{is},e^{it}])=e^{is}\otimes e^{it}\otimes \ldots\otimes e^{it}.
\end{align*}

Let $V_i\to B_i$ be the complex line bundle having $C_i\to B_i$ as unit-circle sub-bundle. By the remark in the previous paragraph, the principal $S^1$-bundle $C_1^q\times_{S^1}C_2\to B_1\times B_2$ is isomorphic to the unit-circle sub-bundle of the complex line bundle
\begin{align*}
 V_1\otimes\overline{V_2}{}^{\otimes q}\to B_1\times B_2.
\end{align*}
Finally, we recall that the Euler class of a unit-circle sub-bundle is the first Chern class of the ambient complex line bundle. Since the first Chern class changes sign under complex conjugation of the bundle, and is additive under tensor products of bundles, we readily obtain our assertion.
\end{proof}

\bibliography{_biblio}
\bibliographystyle{amsalpha}

\end{document}